\documentclass[a4paper]{amsart}
\usepackage{amssymb, amsmath}
\usepackage[all]{xy}

\theoremstyle{plain} 
\newtheorem{theorem}{Theorem}[section]
\newtheorem{thm}[theorem]{Theorem}
\newtheorem{lemma}[theorem]{Lemma}
\newtheorem{lem}[theorem]{Lemma}

\newtheorem{proposition}[theorem]{Proposition}  
\newtheorem{pro}[theorem]{Proposition}
 
\newtheorem{cor}[theorem]{Corollary} 

\newtheorem{eg}[theorem]{Example}

\newtheorem{defn}[theorem]{Definition} 

\newcommand{\up}{\textup}
\newcommand{\NP}{\textup{\texttt{NP}}}
\newcommand{\NL}{\textup{\texttt{NL}}}
\newcommand{\Ll}{\textup{\texttt{L}}}
\newcommand{\Poly}{\textup{\texttt{P}}}

\newcommand{\CSP}{\operatorname{CSP}}

\newcommand{\Mod}{\operatorname{Mod}}
\newcommand{\SD}{\operatorname{SD}}
\newcommand{\dotcup}{\mathbin{\dot{\cup}}}
\newcommand{\structcup}{\mathbin{\overline{\cup}}}

\def\ra{\rightarrow}

\def\la{\leftarrow}

\begin{document}
\title
[Complexity and polymorphisms for digraph constraint problems]{Complexity and polymorphisms for digraph constraint problems under some basic constructions}

\author{Marcel Jackson}
\address{School of Engineering and Mathematical Sciences, La Trobe University, VIC 3086, Australia}
\email{m.g.jackson@latrobe.edu.au}

\author{Tomasz Kowalski}
\address{School of Engineering and Mathematical Sciences, La Trobe University, VIC 3086, Australia}
\email{t.kowalski@latrobe.edu.au}

\author{Todd Niven}
\address{School of Engineering and Mathematical Sciences, La Trobe University, VIC 3086, Australia}
\email{todd.niven@gmail.com}

\thanks{}
\begin{abstract}
The role of polymorphisms in determining the complexity of constraint satisfaction problems is well established.
  In this context we study the stability of CSP complexity and polymorphism properties under some basic graph theoretic constructions.  As applications we observe a collapse in the applicability of algorithms for CSPs over directed graphs with both a total source and a total sink: the corresponding CSP is solvable by the ``few subpowers algorithm'' if and only if it is solvable by a local consistency check algorithm.  Moreover, we find that the property of ``strict width'' and solvability by few subpowers are unstable under first order reductions.  The analysis also yields a complete characterisation of the main polymorphism properties for digraphs whose symmetric closure is a complete graph.
 \end{abstract}
 \thanks{The first and third authors were supported by ARC Discovery Project DP1094578.  The first author was also supported by ARC Future Fellowship FT120100666 and the second author was supported by ARC Future Fellowship FT100100952.}
\maketitle

The influential \emph{Dichotomy Conjecture} of Feder and Vardi \cite{fedvar} proposes that every constraint satisfaction problem over a fixed finite template is either solvable in polynomial time or \NP-complete.  Some well known dichotomies pre-dating the conjecture are special cases: Schaefer's dichotomy for the complexity of CSPs over two-element templates \cite{sch} and the Hell-Ne\v{s}et\v{r}il dichotomy for graph colouring problems \cite{helnes:90} (which is equivalent to the CSP dichotomy conjecture in the case of simple graph templates). Since Feder and Vardi's seminal contribution, the dichotomy has been established for a wide variety of other restricted cases: some of the broadest cases are the dichotomy for three-element templates (Bulatov \cite{bul3}), for list homomorphism problems (also known as conservative CSPs; Bulatov \cite{bulLH}) and for directed graphs with no sources or sinks (Barto, Kozik, Niven \cite{BKN}).

A key tool in more recent advances, including each of \cite{BKN,bul3,bulLH}, has been the universal-algebraic and combinatorial analysis of ``polymorphisms'' of CSP templates.  Polymorphisms are a generalisation of endomorphisms, and give a CSP template a kind of algebraic structure.  
When $\mathbb{A}$ has polymorphisms satisfying certain equational properties, then  $\CSP(\mathbb{A})$ is amenable to tractable algorithmic solution; this is already present in Feder and Vardi \cite{fedvar} for example, where a solution by local consistency check algorithm is shown to hold in the presence of polymorphisms witnessing what are known as ``near unanimity'' equations.  In the other direction, the failure of $\mathbb{A}$ to have polymorphisms satisfying some families of equations can be used to deduce hardness results for $\CSP(\mathbb{A})$; see Larose and Tesson \cite[Theorem~4.1]{lartes} for example.  The polymorphism structure of a CSP template $\mathbb{A}$ has an extremely tight relationship with the complexity of $\CSP(\mathbb{A})$ and appears to relate to the precise structure of algorithms; for example, it is known that solvability by a local consistency check algorithm is equivalent to the presence of specific polymorphism properties (Barto and Kozik \cite{b-k2,barkoz}).

CSPs over directed graphs (henceforth, digraphs) are particularly pertinent to the investigation of CSP complexity.  First, digraph homomorphism properties are a popular topic in their own right (see Hell and Ne\v{s}et\v{r}il's book \cite{helnes} for example).  Second, Feder and Vardi showed that the CSP dichotomy conjecture in full generality is equivalent to its restriction to digraph CSPs.  Moreover it is now known (Bulin, Delic, Jackson and Niven \cite{BDJN,BDJN2}) that the relationship between digraph CSPs and general fixed template CSPs is much tighter than that.  For every template~$\mathbb{A}$ one can associate a directed graph $\mathcal{D}(\mathbb{A})$ such that $\CSP(\mathbb{A})$ and $\CSP(\mathcal{D}(\mathbb{A}))$ are equivalent up to logspace reductions, and such that most polymorphism properties (with the exception of having a ``Maltsev polymorphism'') are held equivalently on $\mathbb{A}$ and $\mathcal{D}(\mathbb{A})$.  In particular, it follows that consideration of finer level computational complexity issues, (such as separation of complexity classes within \texttt{P}) can also be restricted to directed graphs with little or no loss of generality.

In the present article we examine the stability of polymorphism properties of (finite) digraph templates under some basic digraph-theoretic constructions.  The main focus is ``one-point extensions'', meaning the addition of a new total source (or sink), however we also make some easier observations of constructions such as disjoint union and direct product.  Even from this limited selection of constructions there is some interesting behaviour to be found, and the investigation yields a number of consequences.  
\begin{enumerate}
\item Our analysis enables quite easy explicit construction of digraph CSP templates separating essentially all of the currently investigated polymorphism properties.
\item By applying our constructions to the $\mathcal{D}(\mathbb{A})$ construction of \cite{BDJN,BDJN2}, we can show that there is a different translation from general fixed template CSPs to digraph CSPs which maintains logspace-equivalence of the computational problems yet \emph{fails} to preserve a wide range of polymorphism problems, as well as some algorithmic properties (such as strict width).  After the initial results of \cite{BDJN} were proved in 2011, the question as to whether polymorphism properties were always preserved by such translations was of particular interest.  The results of the present article were obtained before publication of \cite{BDJN} and are referenced there and in the expanded version \cite{BDJN2}.
\item Following Kazda's proof that digraph CSPs with a Maltsev polymorphism also have a majority polymorphism \cite{kaz}, there has been  interest in finding classes of digraphs satisfying other collapses of polymorphism properties: see Mar\'oti and Z\'adori \cite{MZ} for example.  A different interpretation of consequence 2 above is that the class of digraphs with a total source and total sink exhibits very strong collapses in polymorphism properties (including the ones identified in \cite{MZ}).  The class of digraphs with total source and total sink has the extra property that it exhibits the full spectrum of CSP complexity:  every digraph CSP is equivalent under first order reductions to one with a total source and total sink.
\item A corollary of our analysis is a precise polymorphism and computational complexity theoretic characterisation of CSPs over semicomplete digraphs: finite digraphs whose symmetric closure is a complete graph (this includes for example, all tournaments).  The classification of tractable CSPs over semicomplete digraphs is due to Bang-Jensen, Hell and MacGillivray~\cite{BJHM}, however our contribution provides a precise fine-level classification in terms of complexity classes and polymorphisms.
\end{enumerate}

The structure of the article is as follows.  The first few sections are intended to serve as a short survey of background information on computational complexity and the universal algebra of constraint problems.  Section \ref{sec:pre} contains basic information on structures and computational complexity.    Section \ref{sec:poly} continues with background information by presenting an overview of the most commonly encountered polymorphism properties and the known and conjectured relationships with complexity classes and algorithms.   Section \ref{sec:reddig} concerns basic definitions and constructions on directed graphs and the stability of membership in complexity classes under these constructions.  Section \ref{sec:unionspoly} investigates some initial results on polymorphisms and basic equational properties.  

Section \ref{sec:npermtopbot} investigates stability of some more complicated polymorphism properties across one point extensions of core digraphs.  The specific focus is the simultaneous extension by both a total source and total sink, which we find pushes up the value of various parameters relating to the length of term conditions.  As an example, the case of transitive tournaments is given a precise polymorphism ``classification''.  

Section \ref{sec:permutation} then examines one-sided extensions by a total source or a total sink: these can exhibit different behaviour than what happens under two sided extensions, depending on the structure of the digraph to which the construction is being applied.  As an example we classify some of the polymorphism properties of digraphs obtained by successively taking  one-sided extensions of directed cycles (not necessarily on the same ``side''). 

In Section \ref{sec:congdist} we show that two of the most commonly encountered polymorphism properties are not necessarily preserved under taking one-point extensions.  We show that a digraph with both a total source and total sink has substantially limited polymorphism behaviour: the existence of polymorphisms for congruence modularity implies the existence of a near unanimity polymorphism.  The main results are Theorem \ref{thm:collapse} which gives the results described in items 2 and 3 above.

The final sections contain some easy applications of the earlier results.  Section \ref{sec:precomplete} combines examples in the article with the ``few sources and sinks'' theorem of Barto, Kozik and Niven \cite{BKN} to verify the algebraic dichotomy conjecture for digraphs whose symmetric closure is a complete graph: we give a precise polymorphism and computational complexity classification (item 4 above).  In Section \ref{sec:separation} we use our investigations to provide instances of CSPs separating all of the commonly encountered polymorphism properties.  We also make some observations relating to the possible complexities of fixed template CSPs, including the observation that (subject to reasonable complexity theoretic assumptions) there must be more possibilities than first order definability and completeness in one of the classes $\Ll$, $\NL$, $\Mod_p\Ll$ (for prime $p$), $\texttt{P}$ and $\texttt{NP}$.

\section{Preliminaries: Structures and complexity}\label{sec:pre}
\subsection{Relational structures and CSPs}  Here we list elementary definitions relating to relational structures and constraint satisfaction problems.  Aside from formalising some notational conventions, readers familiar with the area can safely skip this subsection.

A \emph{relation} $R$ on a set $A$ is a subset of some finite cartesian power $A^k$ of $A$; so $R$ is a set of $k$-tuples of elements of $A$.  The number $k\in\mathbb{N}$ is the \emph{arity} of $R$ and we say that $R$ is a \emph{$k$-ary relation}.  The edge relation of a digraph is an instance of a relation of arity $2$ (that is, a binary relation).  A finite \emph{relational signature} $\mathcal{R}$ is a family of symbols $R_1,\dots,R_n$, each with an associated number in $\mathbb{N}$, the \emph{arity} of the symbol.  An interpretation of $\mathcal{R}$ is a function from the symbols in $\mathcal{R}$ to relations on $A$ that preserves arity; we let $\mathcal{R}_A$ denote an interpretation of $\mathcal{R}$ as relations on $A$, and for each $R\in\mathcal{R}$ we let $R_A$ denote the interpretation of the relation symbol $R$ on $A$.  An interpretation of $\mathcal{R}$ makes $A$ a relational structure of signature $\mathcal{R}$---or an \emph{$\mathcal{R}$-structure}---which we denote in boldface as $\mathbb{A}=\langle A;\mathcal{R}_{A}\rangle$.  The case where $\mathbb{A}$ is a digraph corresponds to the situation where $\mathcal{R}$ consists of a single binary relation symbol; the set $A$ simply corresponds to the vertices of the digraph, and the binary relation symbol is interpreted as the set of directed edges.  We frequently use graph theoretic notation such as $G=(V,E)$ to denote a digraph on vertices $V$ and edges $E$ (rather than, say $\mathbb{V}=\langle V,E\rangle$).  We will also often write $E(u,v)$ or $u\rightarrow v$ to denote $(u,v)\in E$.

A \emph{homomorphism} $f:\mathbb{A}\to\mathbb{B}$ between two structures $\mathbb{A}$ and $\mathbb{B}$ of the same relational signature is a function $f:A\to B$ that preserves each relation: if $(a_1,\dots,a_k)\in R_A$, then $(f(a_1),\dots,f(a_k))\in R_B$.  This notion of homomorphism coincides with the standard definition of a digraph homomorphism, such as is treated in the book by Hell and Ne\v{s}et\v{r}il \cite{helnes}.   A homomorphism $f:\mathbb{A}\to\mathbb{A}$ is called an \emph{endomorphism}.  It is an \emph{automorphism} if $f$ is an isomorphism, and is a retraction if $f\circ f=f$.  The image of a retraction is said to be a \emph{retract} of $\mathbb{A}$.  A structure for which the only endomorphisms are automorphisms is known as a \emph{core}.  Every finite relational structure has a core retract, unique up to isomorphism.

The \emph{direct product} of a finite family of relational structures $\mathbb{A}_1,\dots,\mathbb{A}_k$ of the same relational signature is the relational structure on the cartesian product $A_1\times\dots\times A_k$ with each relation $R$ (of arity $m$, say) interpreted pointwise: if $\mathbb{A}$ denotes $\mathbb{A}_1\times\dots\times\mathbb{A}_k$ and each $\vec{a}_i$ denotes some element $(a_{i,1},\dots,a_{i,k})\in A_1\times\dots\times A_k$, then $(\vec{a}_1,\dots,\vec{a}_m)\in R_A$ when $(a_{1,j},\dots,a_{m,j})\in R_{A_j}$ for each $j=1,\dots,k$. 

A \emph{polymorphism} of $\mathbb{A}$ is a homomorphism $f:\mathbb{A}^k\to\mathbb{A}$; here $k$ is the \emph{arity} of the polymorphism.

The disjoint union $\mathbb{A}\dotcup\mathbb{B}$ of two relational structures $\mathbb{A}$ and $\mathbb{B}$ of the same signature $\mathcal{R}$ is the structure on the disjoint union $A\mathbin{\dot\cup} B$ with each $R\in \mathcal{R}$ interpreted as $R_A\cup R_B$ (again, a disjoint union).  We also make use of a second variation of disjoint union, which in addition to the fundamental relations on $\mathbb{A}\cup\mathbb{B}$ is given two new unary relations $u_A$ and $u_B$, interpreted as $A$ and $B$ respectively.  We refer to this is the \emph{structured union} and denote it by $\mathbb{A}\mathbin{\overline{\cup}}\mathbb{B}$.  

The \emph{constraint satisfaction problem} $\CSP(\mathbb{A})$ is the class of finite structures admitting a homomorphism into $\mathbb{A}$: that is, $\{\mathbb{B}\mid \exists f:\mathbb{B}\to \mathbb{A}\}$.  It is implicit that these structures are of the same signature.  Subject to considering sensible representatives of each isomorphism class in $\CSP(\mathbb{A})$, we can also think of $\CSP(\mathbb{A})$ as a computational problem (corresponding to deciding membership in $\{\mathbb{B}\mid \exists f:\mathbb{B}\to \mathbb{A}\}$) and move freely between these interpretations.  If $\mathbb{A}'$ is the core retract of $\mathbb{A}$, then $\CSP(\mathbb{A})=\CSP(\mathbb{A}')$, and for this reason we frequently restrict our attention to core structures.  Note that when $(V,E)$ is a digraph with a loop, then it has a one element retract so is a trivial CSP. \ For this reason we will assume that our digraphs have no loops; only some of the arguments require this assumption.  For digraphs $G$, the problem $\CSP(G)$ is sometimes alternatively referred to as the $G$-colouring problem.

\subsection{Computational complexity}
We assume basic familiarity with computational complexity, yet to achieve some level of  completeness and to fix notation, we now recall a few notions of particular importance to the article.  For a more general introduction to computational complexity, including time complexity, space complexity, many-one reductions and Turing reductions, see a text such as Papadimitriou \cite{pap}.  For a logic-based approach see Immerman \cite{imm}.  

We treat our decision problems as language membership problems: that is, deciding membership of input words in the language of YES instances.  In each of the following definitions there is no loss of generality in assuming that our Turing machines have a read-only input, a write-only output (if required) and a working tape.  Space is measured on the working tape and by setting alarm clocks if necessary, there is also no loss of generality in assuming that our Turing machines always halt in the given time or space complexity.  An \emph{accepting computation} means one leading to a pre-designated accepting state, while \emph{acceptance} of an input means that there exists an accepting computation.
\begin{itemize}
\item \texttt{NP}: acceptance in $O(n^c)$ time by a nondeterministic Turing machine.
\item \texttt{P}:  acceptance in $O(n^c)$ time by a deterministic Turing machine.
\item \texttt{L}: acceptance in $O(\log_2(n))$ space by a deterministic Turing machine.
\item \texttt{NL}: acceptance in $O(\log_2(n))$ space  by a nondeterministic Turing machine.
\item $\texttt{Mod}_k(\texttt{L})$ (for $k\in\{2,3,\dots\}$): there is a nondeterministic Turing machine running in $O(\log_2(n))$ space and such that the NO instances are those for which the number of accepting computations is equivalent to $0$ modulo $k$.
\end{itemize} 
The last of these classes is the least familiar, yet arises naturally in problems relating to linear algebra; see  Buntrock et al.~\cite{BDHM}, where it is shown that many of the basic computational problems for the linear algebra of a finite field are complete in $\Mod_k(\texttt{L})$ for relevant $k$.  The $\Mod_k(\texttt{L})$ classes also arise in the context of constraint satisfaction problems in  \cite{ABISV} and \cite{lartes}.  The case when $k=2$ is usually denoted $\oplus\texttt{L}$ and called \emph{parity $L$}.  

The following basic containments are well known:
\[
\texttt{L}\subseteq\left\{\begin{matrix}\texttt{NL}\\
\forall k\ \Mod_k(\texttt{L})
\end{matrix}\right\}\subseteq\texttt{P}\subseteq\texttt{NP}.
\]
The precise relationship between the classes $\Mod_k(\texttt{L})$ for various $k$, and of these classes with \texttt{NL} appears to be unresolved.   There are some technical collapses within the family of classes $\{\Mod_k\Ll\mid k>1\}$: in particular, if the number $k$ is written as a product of distinct prime powers $p_1^{k_1}\dots p_n^{k_n}$, then $\Mod_{k}\Ll$ coincides with $\Mod_{p_1\dots p_n}\Ll$; see Buntrock et al.~\cite{BDHM}.  The following facts are also known 
(where $\Ll^{\Mod_p\Ll}$ refers to the class $\texttt{L}$ as defined on oracles for $\Mod_p\Ll$ languages: see \cite[\S14.3, \S17.1]{pap} for general discussion of this concept and notation).
\begin{lem}\label{lem:modpl}\hfill
\begin{itemize}
\item \up(\cite{BDHM}\up) $\Mod_k\Ll\subseteq \Mod_n\Ll$ if $k$ divides $n$.
\item \up(\cite{BDHM}\up) For all $k>1$, the class $\Mod_k\Ll$ is closed under unions of languages.
\item \up(\cite{BDHM}\up) If $p$ is prime then $\Mod_p\Ll$ is closed under intersections of languages.
\item \up(\cite{HRV}\up) If $p$ is prime then $\Mod_p\Ll$ is closed under logspace Turing reductions; so $\Ll^{\Mod_p\Ll}=\Mod_p\Ll$.
\end{itemize}
\end{lem}
It is currently unknown if there is a containment between $\NL$ and $\Mod_k(\texttt{L})$ for some~$k$.  
Similarly, at the time of writing, it remains a possibility that $\Mod_m(\texttt{L})$ and $\Mod_n(\texttt{L})$  are incomparable unless the set of prime factors of $m$ is a subset of the set of prime factors of $n$, or vice versa. It is also appears unknown whether $\Mod_k(\texttt{L})$ is in general closed under intersection.  It is, of course, also unknown if any of the containments mentioned above are strict.

Finally for this subsection, we remind the reader of the definition of a first order reduction.  As is explained in Immerman \cite{imm}, it is convenient to assume throughout that our structures come with some predetermined linear order $<$.  

Fix two relational signatures $\mathcal{R}$ and $\mathcal{S}$.  For variables $\vec{x}=x_0,\dots,x_{n-1}$ and $\vec{y}=y_0,\dots,y_{p-1}$ a first order $\mathcal{R}\cup\{<\}$-formula $\phi(\vec{x},\vec{y})$ determines a parameterised family of $n$-ary relations on any $\mathcal{R}\cup\{<\}$-structure $\mathbb{A}$: for any $\vec{a}=a_0,\dots,a_{p-1}\in A$ (the parameters), the solution set of $\phi(\vec{x},\vec{a})$ is an $n$-ary relation.  Similarly, for any fundamental relation $S\in\mathcal{S}$ of arity $k\in\mathbb{N}$, a $(kn+p)$-ary formula $\psi_S(\vec{x_1},\dots,\vec{x_{k}},\vec{y})$ (where $\vec{x_i}$ denotes $x_{i,0},\dots,x_{i,{n-1}}$) determines a parameterised family of $k$-ary relations on $n$-tuples: for any interpretation $\vec{a}$ of $\vec{y}$ in $A$ we obtain an $k$-ary relation on $A^n$.    Then, for any evaluation of $\vec{y}$ as $\vec{a}$ in $A$, the family
\[
\{\phi(\vec{x},\vec{a})\}\cup \{\psi_S(\vec{x_1},\dots,\vec{x_{k}},\vec{a})\mid S\in\mathcal{S}\}
\]
defines an $\mathcal{S}$-structure $\mu_{\vec{a}}(\mathbb{A})$, whose universe $U$ is the solution set of $\phi(\vec{x},\vec{a})$, as a subset of $A^n$, and where the relations on $\mu_{\vec{a}}(\mathbb{A})$ are the restriction to $U$ of the solution sets of $\psi_S(\vec{x_1},\dots,\vec{x_{k}},\vec{a})$, for each $S\in\mathcal{S}$.  
These formul{\ae} form an \emph{$n$-ary first order reduction with $p$-parameters} from a decision problem $\Pi_1$  on $\mathcal{R}$-structures to a decision problem $\Pi_2$ on $\mathcal{S}$-structures if for every $\mathcal{R}$-structure $\mathbb{A}$ with at least $p$ elements and every $p$-tuple of \emph{pairwise distinct} elements $a_0,\dots,a_{p-1}\in A$ we have $\mathbb{A}\in\Pi_1$ if and only if $\mu_{\vec{a}}(\mathbb{A})\in \Pi_2$.  Note that a linear order $<$ on $\mu_{\vec{a}}(\mathbb{A})$ can always be easily definable in terms of $<$ on $\mathbb{A}$.

First order reductions are computable in logspace, and thus provide a finer separation tool for comparing computational complexity than logspace (or worse, polynomial time) many-one reductions.  In particular, they enable a meaningful notion of completeness for classes such as logspace and nondeterministic logspace; see \cite{imm} for example.  Each of the complexity classes mentioned above is closed under logspace many-one reductions and hence first order reductions.

\section{Polymorphisms, complexity and algorithms}\label{sec:poly}
  Readers familiar with the role of universal algebra in CSP complexity can skip this section, 
though they may want to briefly refer to the precise formulations of various term conditions we present for use in the article.  The methods of the article will not require a knowledge of these universal algebraic methods, however their role in CSP complexity is too deeply entwined for the article to avoid some elaboration of the basic connections.  In particular, the kinds of polymorphism properties we study throughout the article are of interest precisely because of the fact that they correspond to natural boundaries suggested by universal algebraic concepts.  In this section we present a broad overview of the connections: the presentation is not intended to be encyclopaedic, but rather enough to motivate the concepts and to give meaning to the definitions we require.

The role of polymorphisms in determining the computational complexity of (membership in) $\CSP(\mathbb{A})$ is very precise: if the underlying sets of $\mathbb{A}$ and $\mathbb{B}$ coincide and if every polymorphism of $\mathbb{B}$ is a polymorphism of $\mathbb{A}$ then $\CSP(\mathbb{B})$ reduces to $\CSP(\mathbb{A})$.  Ideas introduced in Jeavons \cite{jea} and expanded in Bulatov, Jeavons and Krokhin \cite{BJK} showed that this can be extended substantially further.  Recall that if $\mathbf{A}$ is an algebraic structure, then the \emph{variety} generated by $\mathbf{A}$ is the class of algebraic structures arising as homomorphic images of subalgebras of direct products of $\mathbf{A}$; equivalently, it is the class of all algebras (with the same kinds of operations) satisfying all identities of $\mathbf{A}$.  If the set $A$ is endowed with the family of all polymorphisms of $\mathbb{A}$, then $A$ becomes an algebraic structure $\mathbf{A}$ (of infinite type), which we refer to as the \emph{polymorphism algebra of $\mathbb{A}$}.  The following makes use of the well-known Galois connection between clones of operations and clones of relations, as generated by primitive positive formul{\ae} (\emph{pp-formul{\ae}}); it is essentially a presentation of results in \cite{BJK}, with some refinements from Larose and Tesson \cite{lartes}.  
We use $\operatorname{Inv}$ to denote the relations invariant under the operations of an algebra, and $\operatorname{Clo}$ to denote the relational clone generated by a set of relations using pp-formul{\ae}.
\begin{thm}\label{thm:hsp} \up(Bulatov, Jeavons, Krokhin \cite{BJK}, Larose, Tesson \cite{lartes}.\up)
Let $\mathbb{A}=\langle A,\mathcal{R}_{A}\rangle$ and $\mathbb{B}=\langle B,\mathcal{S}_{B}\rangle$ be relational structures.  If ${\bf A}$ is an algebra on the set $A$ and ${\bf B}$ is an algebra on the set $B$ such that 
\begin{enumerate}
\item $\mathbf{B}$ is contained in the variety of $\mathbf{A}$,
\item every term function of $\mathbf{B}$ is a polymorphism of $\mathbb{B}$,
\item every polymorphism of $\mathbb{A}$ is a term function of ${\bf A}$
\end{enumerate}
 then for any finite subset $\mathcal{S}_{B}'\subseteq \mathcal{S}_{B}$ 
 there is a finite subset $\mathcal{R}_{A}'\subseteq \mathcal{R}_{A}$ such that  $\CSP(\langle B,\mathcal{S}_{B}'\rangle)$ reduces in logspace to $\CSP(\langle A,\mathcal{R}_{A}'\rangle)$.  In particular, if $\mathbb{B}$ has finite signature, then $\CSP(\mathbb{B})$ reduces in logspace to $\CSP(\mathbb{A})$.
\end{thm}
\begin{proof}[Proof sketch]
We give references to both \cite{lartes} and \cite{BJK} as appropriate.   By (2), $\mathcal{S}_{B}'$ is a finite subset of $\operatorname{Inv}(\mathbf{B})$.  By (1) there is ${\bf C}$ and $n$ with $\exists \phi:{\bf C}\twoheadrightarrow {\bf B}$ and ${\bf C}\leq {\bf A}^n$.  Using $\phi$, a first order reduction can be found from $\CSP(\mathcal{S}_{B}')$ to a CSP over some finite subset $\mathcal{T}_{\bf C}\subseteq \operatorname{Inv}(\mathbf{C})$; see \cite[Theorem 5.4]{BJK} and \cite[Lemma~2.2]{lartes}.   Using ${\bf C}\leq {\bf A}^n$, we have $\mathcal{T}_{\bf C}\subseteq \operatorname{Inv}(\mathbf{A}^n)$.  Next, there is a first order reduction from the CSP over any finite subset of $\operatorname{Inv}(\mathbf{A}^n)$ to the CSP over a suitable finite subset of $\operatorname{Inv}(\mathbf{A})$; in particular, from $\CSP(\langle A^n,\mathcal{T}_C\rangle)$ to the CSP over some finite subset $\mathcal{U}_{A}$ of $\operatorname{Inv}(\mathbf{A})$.  (This reduction is by repeated application of \cite[Lemma~2.4]{lartes}, and is not given in \cite{BJK}.)  By (3) we have that $\operatorname{Inv}(\mathbf{A})$ is a subset of $\operatorname{Clo}(\mathbb{A})$ and so $\mathcal{U}_{A}$ is pp-definable from the relations $\mathcal{R}_{A}$ of $\mathbb{A}$.  Because pp-formul{\ae} involve only finitely many relations, this provides a logspace reduction from $\CSP(\langle A,\mathcal{U}_{A}\rangle)$ to $\CSP(\langle A,\mathcal{R}_{A}'\rangle)$, for some finite subset $\mathcal{R}_{A}'$ of $\mathcal{R}_{A}$, as required; see Jeavons \cite[Corollary 4.1]{jea} and Lemmas 2.8--2.11 of \cite{lartes}, though pp-reductions go back at least to Schaefer \cite{sch}.  When $\mathbb{B}$ is of finite signature, start with $\mathcal{S}_{B}'=\mathcal{S}_{B}$.
\end{proof}
 This gave a new perspective on CSP complexity: if $\CSP(\mathbb{B})$ is complete in some complexity class $\mathcal{K}$, then for $\CSP(\mathbb{A})$ to avoid being $\mathcal{K}$-hard, the polymorphism algebra of  $\mathbb{A}$ must avoid algebras whose term functions are amongst the polymorphisms of $\mathbb{B}$.  
Conveniently, a number of well known CSP problems turn out to correspond to basic kinds of algebraic structure.  In the following list (whose enumeration with missing (iii) will become clear in due course), we recall some polymorphism algebras of classic computational problems and the consequences of Theorem \ref{thm:hsp}.  A far more rigorous treatment can be found in the work of Larose and Tesson \cite{lartes}, who have shown that in general, the critical ``hardness'' statements of relevance can be made in terms of first order reductions rather than simply logspace reductions (this is evident in the proof sketch: only the removal of equality constraints requires logspace).
\begin{itemize}
\item[(i)] The polymorphisms for classic $\NP$-complete  problems such as 3SAT and Graph 3-colourability are essentially degenerate: they consist of projections, automorphisms and their compositions.  In the universal algebra lingo, these are known as algebras of ``unary type''. \emph{To avoid being \NP-complete, the polymorphism algebra of a template must generate a variety avoiding unary type.}
\item[(ii)] solvability of linear equations over a field of order $p$ (the archetypal $\Mod_p\Ll$-complete problem) has polymorphisms the same as those of a module over $\mathbb{Z}_p$.  \emph{To avoid being $\Mod_p\Ll$-hard for some $p$, the polymorphism algebra of a template must generate a variety avoiding module-type algebras.}
\item[(iv)] directed graph unreachability (an $\NL$-complete CSP) has  polymorphisms the same as the term functions of a lattice.  \emph{To avoid being $\NL$-hard, the polymorphism algebra of a template must generate a variety avoiding lattice-like algebras.}
\item[(v)] The polymorphisms of HORN3SAT (which is $\Poly$-complete) coincide with the term functions of the two element semilattice.  \emph{To avoid being $\Poly$-hard, the polymorphism algebra of a template must generate a variety avoiding semilattice-like algebras.}
\end{itemize}

The first instance in this list---avoiding algebras of unary type---is one half of the \emph{algebraic dichotomy conjecture}, which states that for a core CSP template~$\mathbb{A}$, the problem~$\CSP(\mathbb{A})$ is $\NP$-complete if the variety generated by the polymorphism algebra contains algebras of unary type, and otherwise is tractable.  

The analysis of which properties guarantee that a finitely generated variety avoids objects of unary type, semilattice type, and so on, is a substantially developed theme within universal algebra.  \emph{Tame congruence theory} \cite{hobmck}, has focussed on the presence of objects that fall into one of the following 5 types (in their standard numbering):
\begin{enumerate}
\item Unary type;
\item Affine type (modules);
\item Boolean type;
\item Lattice type;
\item Semilattice type.
\end{enumerate}
All except Boolean type were loosely introduced by way of observations on CSP-hardness in discussion above: each of items (i), (ii), (iv) and (v) correspond to their respective arabic numeral type.  The missing case of Boolean type (type 3) corresponds to trivial CSP problems and so does not play a role in establishing hardness results; we omit further discussion of this type of structure.

Figure \ref{fig:maltsev} presents a diagram of what might be called the ``universal algebraic geography of CSPs'' (it is slightly simplified from the full picture for universal algebras).
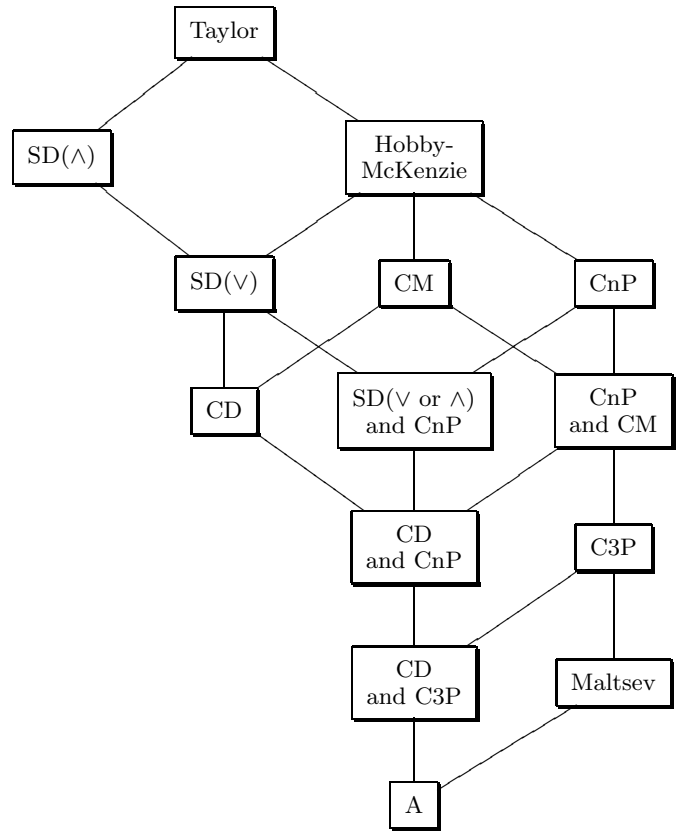
\begin{figure}[h]
  \begin{center}
    \begin{small}
      \xymatrix{ &&& *++[F-,]\txt{Taylor}\ar@{-}[rd] &
        \\
        &&*++[F-,]{\txt{SD$(\wedge)$}}\ar@{-}[ur] &
        &*++[F-,]{\txt{Hobby- \\ McKenzie}} &
        \\
        &&&*++[F-,]{\txt{SD$(\vee)$}}\ar@{-}[ur]\ar@{-}[ul] &
        *++[F-,]\txt{CM}\ar@{-}[u] & *++[F-,]\txt{CnP}\ar@{-}[ul] &
        \\
        &&&*++[F-,]\txt{CD}\ar@{-}[u]\ar@{-}[ur] &
        *++[F-,]{\txt{SD$(\vee\text{ or }\wedge)$ \\ and CnP}}\ar@{-}[ul]\ar@{-}[ur]
        & *++[F-,]\txt{CnP \\ and CM}\ar@{-}[u]\ar@{-}[ul] &
        \\
        &&& & *++[F-,]\txt{CD \\ and
          CnP}\ar@{-}[u]\ar@{-}[ur]\ar@{-}[ul] &
        *++[F-,]\txt{C3P}\ar@{-}[u]
        \\
        &&& & *++[F-,]\txt{CD \\ and C3P}\ar@{-}[u]\ar@{-}[ur] &
        *++[F-,]\txt{Maltsev}\ar@{-}[u]
        \\
        &&& & *++[F-,]\txt{A}\ar@{-}[u]\ar@{-}[ur] & }
    \end{small}
  \end{center}\caption{The universal algebraic geography of tractable CSPs.}\label{fig:maltsev}
\end{figure}
The remainder of this section is devoted to  giving meaning to Figure~\ref{fig:maltsev}, by describing the regions in the picture with reference to the five types just described and providing the term conditions that they correspond to.   All polymorphism conditions characterising the properties in Figure \ref{fig:maltsev} will be in terms of idempotent polymorphisms: a polymorphism $p$ satisfying $p(x,x\dots,x)=x$ (for all $x$).  
The nodes equal to or above ``$\text{SD}(\vee\text{ or }\wedge)$ and $\text{C}n\text{P}$'' are referred to as the \emph{upper level}, with the remaining nodes form the \emph{lower levels}.  

As a general rule, the \emph{node labels indicate a property that is assumed to hold at that point and in all lower cases}.  We say that a CSP lies at one of the nodes if it has the property stated as the node label, but none of the lower properties.  Infimums in the order shown in Figure \ref{fig:maltsev} correspond to intersections of classes: thus $\SD(\vee)$ corresponds to the property of having both the Hobby-McKenzie property and the $\SD(\wedge)$ property.

From an algebraic perspective, the various conditions considered  are examples of Maltsev conditions and the  reader wanting a more detailed treatment is encouraged to consult a general text such as Bergman~\cite[\S4.7,4.8]{ber} or a more specific book such Kearnes and Kiss \cite{keakis} for more information.  The article Kozik, Krokhin, Valeriote and Willard \cite{KKVW} is another useful reference for most of the conditions considered below.

\subsection{Taylor.}  This corresponds to algebras whose variety avoids unary type, 
and takes its name from the term conditions derived in Hobby, McKenzie \cite[Chapter~9]{hobmck} based on the work of  Taylor \cite{tay}.
The \emph{algebraic dichotomy conjecture} states (for core relational structure $\mathbb{A}$) that $\CSP(\mathbb{A})$ is tractable if and only if the variety of the polymorphism algebra $\mathbf{A}$ avoids the unary type: in other words, that it lies somewhere in the diagram of Figure \ref{fig:maltsev}.

We use the following conditions characterising the Taylor property due to Mar\'oti and McKenzie \cite{marmck} (but see also Barto and Kozik \cite{b-k3}).
The polymorphism algebra $\mathbf{A}$ of a finite relational structure $\mathbb{A}$ (with finitely many relations) generates a variety avoiding unary type if and only if $\mathbb{A}$ has an idempotent $n$-ary polymorphism $w(x_1,\dots,x_n)$ (for some $n>1$) satisfying for all $x, y\in A$
\[
w(y,x,\dots ,x,x)=w(x,y,\dots,x,x)=\dots=w(x,x,\dots,y,x)=w(x,x,\dots,x,y).
\]
Such a polymorphism is called a \emph{weak near unanimity} polymorphism, or ``weak~NU''.
Equivalent conditions involving either a pair of ternary polymorphisms, or a single $4$-ary polymorphism can be found in Kearnes, Markovi\'c and McKenzie \cite{KMM}.

\subsection{$\SD(\wedge)$: Congruence meet semidistributivity.} This corresponds to the class of algebras whose variety avoids both the unary type and the affine type.  The title comes from the relationship with congruence lattices of algebras: strictly it is the variety generated by the polymorphism algebra that is said to have the congruence meet semidistributivity property, however we abuse the phrase and will refer to the $\SD(\wedge)$ property even for the relational structure from which the polymorphism algebra arises.  

The $\SD(\wedge)$ property plays a central role in tractable CSPs: it is known (Barto and Kozik \cite{barkoz}) that a CSP over a core structure $\mathbb{A}$ is solvable by local consistency check if and only if the polymorphism algebra of $\mathbb{A}$ has the $\SD(\wedge)$ property.  The polymorphism algebra $\mathbf{A}$ of a finite relational structure $\mathbb{A}$ (with finitely many relations) generates a variety avoiding unary and affine types if and only if $\mathbb{A}$ has a $3$-ary weak NU $w_1(x_1,x_2,x_3)$ and a $4$-ary weak NU $w_2(x_1,x_2,x_3,x_4)$ such that $w_1(y,x,x)=w_2(y,x,x,x)$; 
see Kozik, Krokhin, Valeriote and Willard \cite[Theorem~2.8]{KKVW}.

Two subclasses of $\SD(\wedge)$ not shown in Figure \ref{fig:maltsev} are the properties of carrying \emph{totally symmetric idempotent} (TSI) polymorphisms of all arities and the property of carrying polymorphisms giving $\mathbf{A}$ the structure of a $2$-semilattice.  These are discussed separately in Subsection \ref{subsec:tsi}.  

\subsection{Hobby-McKenzie} This corresponds to the class of algebras whose variety avoids both the unary type and the semilattice type.  The authors are not aware of any conjecture tying this to a computational property, however it is a natural future candidate to play some role, and is useful in the present article because it intersects with the class $\SD(\wedge)$ to give $\SD(\vee)$.  We speculate some possible lines of investigation for this class in the conclusion section of the article.

The following equations are from Hobby and McKenzie \cite{hobmck}.  A finite algebra ${\bf A}$ generates a variety with the Hobby-McKenzie property if and only if there is an $n\geq 0$ and $3$-ary idempotent terms $d_0(x,y,z)$, \dots, $d_n(x,y,z)$, $p(x,y,z)$ and $e_0(x,y,z)$, \dots, $e_n(x,y,z)$ satisfying (for all $x,y,z\in A$):
\begin{align*}
x&=d_0(x,y,z)\qquad&\mbox{and}\qquad&e_n(x,y,z)=z\\
d_i(x,y,y)&=d_{i+1}(x,y,y)\qquad&\mbox{and}\qquad&e_i(x,y,y)=e_{i+1}(x,y,y) \quad \mbox{ for even $i<n$}\\
d_i(x,x,y)&=d_{i+1}(x,x,y)\qquad&\mbox{and}\qquad&e_i(x,x,y)=e_{i+1}(x,x,y)\quad \mbox{ for odd $i<n$}\\
d_n(x,y,y)&=p(x,y,y)\qquad&\mbox{and}\qquad&p(x,x,y)=e_{0}(x,x,y)\\
d_i(x,y,x)&=d_{i+1}(x,y,x)\qquad&\mbox{and}\qquad&e_j(x,y,x)=e_{j+1}(x,y,x)\\
&&&\mbox{ for odd $i<n$ and even $j<n$}
\end{align*}
We will say that a CSP template $\mathbb{A}$ has Hobby-McKenzie polymorphisms if its polymorphism algebra has Hobby-McKenzie terms (equivalently, if there are polymorphisms $d_0(x,y,z)$, \dots, $d_n(x,y,z)$, $p(x,y,z)$ and $e_0(x,y,z)$, \dots, $e_n(x,y,z)$ of $\mathbb{A}$ satisfying the Hobby-McKenzie equations).  We often omit reference to $d_0$ and $e_n$ (which are projections) and start with $x=d_1(x,y,y)$ and so on.

\subsection{$\SD(\vee)$: congruence join-semidistributivity.}  This corresponds to those algebras whose variety avoids the affine type, the semilattice type and the unary type.  This class is conjectured in Larose and Tesson \cite{lartes} to contain precisely the polymorphism algebras of CSPs solvable in nondeterministic logspace.  If the polymorphism algebra $\mathbf{A}$ of a CSP is not in this class, then its variety contains either an algebra of unary type (so $\CSP(\mathbb{A})$ is \NP-complete), or an algebra of semilattice type (and is $\Poly$-hard) or an algebra of affine type (so is $\Mod_p\Ll$-hard).    If $\mathbf{A}$ lies in this class but not lower, then it is known that $\CSP(\mathbb{A})$ is $\NL$-hard.  We do not recall a separate term-equation classification for $\SD(\vee)$ here, as the property is equivalent to verifying both $\SD(\wedge)$ and Hobby-McKenzie. 

\subsection{C$n$P; Congruence $n$-permutability (for some $n\geq 2$).}  This corresponds to those algebras avoiding unary, semilattice and lattice type; the name comes from a property on congruences of algebras in the variety but as usual we abuse the phrase and will refer to a digraph as having the congruence $n$-permutability property if the variety generated by its polymorphism algebra has this property.  The cases $n=3$ and $n=2$ are given separate nodes in the diagram, due to the fact that they imply the congruence modularity property. 

Polymorphism algebras lying at the C$n$P node of Figure \ref{fig:maltsev} (but not lower; so they admit affine type) are from templates for $\Mod_p\Ll$-hard CSPs, for some prime $p$ (Larose and Tesson \cite{lartes}).  
A finite algebra ${\bf A}$ generates a congruence $n$-permutable variety (omitting lattice, semilattice and unary types) if and only if there is an $n\geq 1$ and $3$-ary idempotent terms $p_0(x,y,z),\dots,p_n(x,y,z)$ satisfying (for all $x,y,z\in A$):
\begin{align*}
x&=p_0(x,y,z)\\
p_i(x,x,y)&=p_{i+1}(x,y,y)\\
p_n(x,y,z)&=z.
\end{align*}
As before we frequently omit reference to $p_0$ and $p_n$, instead starting with $x=p_1(x,y,y)$ and ending with $p_{n-1}(x,x,y)=y$.   The notion of ``congruence $n$-permu\-table'' is meaningless for $n=1$, however it is convenient to allow it to be interpreted as meaning ``is a one-element structure''.  This is consistent with the equational characterisation, which are equivalent to $(\forall x)(\forall y)\ x=y$ when $n=1$.

\subsection{$\SD(\vee\text{ or }\wedge)$ and C$n$P}  The last class in the top ``layer'' of Figure \ref{fig:maltsev} is the intersection of earlier classes.  This class (and lower) is also often conjectured to consist the polymorphism algebras of those $\mathbb{A}$ for which $\CSP(\mathbb{A})$ is solvable in logspace.  The ``or'' in $\SD(\vee\text{ or }\wedge)$ is because both $\SD(\wedge)$ and $\SD(\vee)$ intersect with C$n$P to the same class of idempotent algebras.  

\subsection{CD: Congruence distributivity.}  This is a further property related to congruence lattices of algebras across a whole variety.  A finite algebra ${\bf A}$ generates a congruence distributive variety if and only if it has terms $J_0(x,y,z),\dots,J_n(x,y,z)$ satisfying
\begin{align*}
J_0(x,y,z)&=x\\
J_i(x,y,x)&=x\mbox{ for all $i\leq n$}\\
J_i(x,x,y)&=J_{i+1}(x,x,y)\mbox{ for all even $i<n$}\\
J_i(x,y,y)&=J_{i+1}(x,y,y)\mbox{ for all odd $i<n$}\\
J_n(x,y,z)&=z.
\end{align*}
These terms come from J\'onsson \cite{jon} and are usually called \emph{J\'onsson terms}.

For the polymorphism algebras of a finite relational structure (with only finitely many relations), this condition has been shown by Barto \cite{bar:NU} to be equivalent to the presence of a weak NU polymorphism $n(x_1,\dots,x_n)$ such that (in addition to the equations already described for weak NUs) the equation $n(y,x,\dots,x)=x$ holds.   Such terms are known as \emph{near unanimity terms} (and are the origin of the abbreviation NU), but were considered independently by Feder and Vardi \cite{fedvar} who showed that a CSP template has the ``strict width'' property if and only if it has an NU polymorphism.  The strict width property is equivalent to every locally consistent partial solution extending to a full solution.  In general, there are finite algebras generating congruence distributive varieties that do not have NU terms: they only arise as polymorphism algebras over relational structures with infinitely many relations.  Even though this article concerns itself with digraphs (just one relation), there are some instances where it is more convenient to use the general congruence distributivity conditions (and then call on \cite{bar:NU} to deduce results about NU conditions and strict width).

\subsection{CM: congruence modularity}\label{subsec:CM}
This condition is a generalisation of congruence distributivity and is characterised by the existence of ternary terms $s_0(x,y,z)$,\ldots, $s_{2n}(x,y,z)$, $p(x,y,z)$ satisfying 
\begin{align*}
s_0(x,y,z)&=x\\
s_i(x,y,x)&=x\mbox{ for all $i\leq 2n$}\\
s_i(x,y,y)&=s_{i+1}(x,y,y)\mbox{ for all even $i<2n$}\\
s_i(x,y,y)&=s_{i+1}(x,y,y)\mbox{ for all odd $i<2n$}\\
s_{2n}(x,y,y)&=p(x,y,y)\\
p(x,x,y)&=y
\end{align*} 
This characterisation was introduced by Gumm \cite{gum} and the terms $s_0,\dots,s_{2n},p$ are usually referred to as \emph{Gumm terms}; an earlier condition was given by Day~\cite{day}.
Libor Barto has recently shown \cite{bar2016} that for polymorphism algebras of finite relational structures (with finitely many relations), the congruence modularity property is equivalent to the existence of what is known as \emph{edge terms}; this solves a problem widely referred to as \emph{Valeriote's conjecture}.  The result has deep implications for CSP complexity: the presence of edge polymorphisms was shown in \cite{BIMMVW,IMMVW} to completely classify solvability of a CSP by way of a particular kind of algorithm generalising Gaussian elimination, known as the \emph{few subpowers algorithm} \cite{BIMMVW,IMMVW}.  While we make reference to the few subpowers property later, we will not require its precise definition in this article.  The results of this article also do not depend on Barto's solution to Valeriote's conjecture.

\subsection{Lower levels in Figure \ref{fig:maltsev}.}  The lower end of Figure \ref{fig:maltsev} mostly consists of some classes generated by intersecting classes higher up in the picture, as well as some important special cases.    As mentioned above, C$3$P and Maltsev are just special cases of C$n$P (which is really an infinite hierarchy).  It is known that congruence $3$-permutability implies congruence modularity, so the intersection of the class of congruence $n$-permutable varieties and the congruence modular varieties contains the class of congruence $3$-permutable varieties.  Congruence $2$-permutability is usually known as \emph{congruence permutability} and corresponds to the presence of a single term $p(x,y,z)$ satisfying $x=p(x,y,y)=p(y,y,x)$.  Such varieties are usually called \emph{Maltsev varieties}, after Maltsev \cite{mal}.  Maltsev varieties have played an important role in the theory of universal algebra, and also in the early development of new algorithms for CSP solution: the generalisation of Gaussian elimination to Maltsev polymorphisms was a crucial technique used by Bulatov in his extension of Schaefer's Dichotomy to $3$-element templates~\cite{bul3}.

A variety that is both Maltsev and Congruence Distributive is said to be \emph{arithmetical}, so the intersection of these classes has been given the abbreviated title~``A'' in Figure~\ref{fig:maltsev}.  
Pixley \cite{pix} showed that the following easy term condition characterises arithmeticity: there is a ternary polymorphism $m$ satisfying $m(x,y,x)=m(x,y,y)=m(y,y,x)=x$.

\subsection{Totally symmetric idempotent polymorphisms}\label{subsec:tsi} Two other commonly encountered polymorphism properties are the \emph{$2$-semilattice polymorphism} and \emph{totally symmetric idempotent polymorphisms}: these are special cases of SD$(\wedge)$, but adding them to the diagram creates extra ``wings'' and many new intersections to the left of Figure \ref{fig:maltsev}.  

The TSI polymorphism property means that for each~$n$ there is an $n$-ary idempotent polymorphism $p_n(x_1,\dots,x_n)$ such that whenever $x_1,\dots,x_n,y_1,\dots,y_n$ is a list of variables, possibly with repeats, and $\{x_1,\dots,x_n\}=\{y_1,\dots,y_n\}$ then $p(x_1,\dots,x_n)=p(y_1,\dots,y_n)$.  An algebra is a $2$-semilattice if there is a binary term~$\cdot$ satisfying $x\cdot y=y\cdot x$ and $x\cdot (x\cdot y)=x\cdot y$.

Examples based over digraphs built from directed cycles will be revisited throughout the article.
\begin{eg}\label{eg:cycle}
The directed cycle $\mathbb{C}_n$ has a Maltsev polymorphism and a ternary NU polymorphism.  So $\CSP(\mathbb{C}_n)$ lies at the node labelled by ``A''.  If $n$ is odd, $\mathbb{C}_n$ has a $2$-semilattice polymorphism, but if $n$ is even it does not have any commutative binary polymorphism.  More generally, $\mathbb{C}_n$ does not have an $n$-ary totally symmetric polymorphism.
\end{eg}
\begin{proof}
These facts are well known, but give some further elaboration of the concepts in Figure \ref{fig:maltsev} and are useful later in the article.  To define a Maltsev polymorphism, consider a triple $(a,b,c)$.  If $|\{a,b,c\}|=1$ or $3$ then let $p(a,b,c)=a$.  Otherwise, let $p(a,b,c)$ be the minority value (that is, the value that is not repeated).  One may similarly define a ternary NU polymorphism $n(x,y,z)$ in the same way, except that when $|\{a,b,c\}|=2$, we let $n(a,b,c)$ take the majority value (that is, the repeated value).  Such a polymorphism is known as a \emph{majority} polymorphism, and makes the polymorphism algebra of $\mathbb{C}_n$ generate a congruence distributive variety, a special case of congruence meet semidistributivity.
For $2$-semilattice where $n$ is odd, let $m:=\lfloor n/2\rfloor$ and let $[k]_n$ denote the smallest nonnegative integer equivalent to $k$ modulo $n$.  Define the binary $\cdot$ on $\{0,\dots,n-1\}$ by 
\[
i\cdot j:=\begin{cases}
i\text{ if }[j-i]_n\leq m\\
j\text{ otherwise}.
\end{cases}
\]
If $n$ is even, then let $m=n/2$ and consider the value $i$ of $0\cdot m$.  Now $(0,m)\ra(1,m+1)\ra(2,m+2)\ra\dots\ra (m,0)$.  So as $\cdot$ is a polymorphism, we may follow the directed edges from $i$ as $i\ra i+1\ra\dots \ra i+m$ to find $m\cdot 0=i+m$.  But this contradicts commutativity of $\cdot$ and $i=0\cdot m$.

For the final claim simply observe that $(0,1,\dots,n-1)\ra(1,2,\dots,n-1,0)$ so that if $t$ is any $n$-ary polymorphism on $\mathbb{C}_n$ we have $(0,1,\dots,n-1)\ra(1,2,\dots,n-1,0)$ giving $t(0,1,\dots,n-1)\ra t(1,2,\dots,n-1,0)$.  So $t(0,1,\dots,n-1)\neq t(1,2,\dots,n-1,0)$ and $t$ is not totally symmetric.
\end{proof}

\section{Digraphs and constructions.}\label{sec:reddig}
The majority of relational structures in this article will be digraphs.  In this section we present the basic notation and constructions on digraphs that we use in the article and present some basic observations relating the stability of CSP complexity under the constructions.
\begin{defn}\hfill
\begin{enumerate}
\item The \emph{directed cycle} on $\{0,1,\dots,k-1\}$ is denoted by $\mathbb{C}_k$; with edge relation $i\rightarrow i+1$ \up(with addition modulo $k$\up).  
\item The transitive tournament on $\{0,\dots,k-1\}$ is denoted by $\mathbb{T}_k$ and has edge relation equal to the usual strict inequality relation~$<$.  For any integers $i\leq j$
we may represent $\mathbb{T}_{j-i+1}$ on vertices $i,i+1,\dots,j$, with edges
$k\rightarrow \ell$ if $k<\ell$; we denote this representative by
$[i,j]=(\{ i,i+1,\dots, j\},<)$.
\end{enumerate}
\end{defn}

\begin{defn}
Let $G=(V,E)$ be a digraph, then for integers $i\leq 0\leq j$ the
digraph $G^{[i,j]}$ will denote the digraph on vertices $\{
i,{i+1},\dots,{-1}\} \cup V \cup \{ 1,2,\dots, j\}$, \up(we assume that
$V$ and $\{ i,i+1,\dots,j\}$ are disjoint, renaming the elements of $V$ if necessary\up) with edge relation
\begin{align*}
&\left\{ (v_1,v_2)\mid v_1< v_2 \text{ in } [i,j]  \right\}\\
\cup&\left\{ (v_1,v_2)\mid v_1,v_2\in V\text{ and }E(v_1,v_2) \right\}\\
\cup&\left\{ (v_1,v_2)\mid v_1<0 \text{ in $[i,j]$ and } v_2\in V\right\}\\
\cup&\left\{ (v_1,v_2)\mid v_2>0 \text{ in  $[i,j]$ and } v_1\in V\right\}.
\end{align*}
\end{defn}

Intuitively $G^{[i,j]}$ can be thought of as the digraph obtained by
replacing the $0$ vertex in $[i,j]$ with the digraph $G$. In
particular, if $G=( \{ 0\},\varnothing)$ then $G^{[i,j]}$ is the
transitive tournament $[i,j]$, while $G^{[0,0]}:=G$.

For a digraph $G=(V,E)$ we define $G^\top=(V\cup \{ \top\},E\cup \{
(v,\top)\mid v\in V\})$ and $G^\bot=(V\cup \{ \bot\},E\cup \{
(\bot,v)\mid v\in V\})$ where $\top,\bot\notin V$. Observe that
$(G^{[i,j]})^\bot$ is isomorphic to $G^{[i-1,j]}$ and similarly
$(G^{[i,j]})^\top$ is isomorphic to $G^{[i,j+1]}$. 

We often refer to $G^{\bot\top}$ (or equivalently $G^{[-1,1]}$) as the
\emph{two-sided extension} of~$G$.

Let $G = (V,E)$ be a digraph. For a vertex $x\in V$ we
define $x^+ = \{v\in V\mid E(x,v)\}$ and $x^- = \{v\in V\mid
E(v,x)\}$. A vertex $a$ is a \emph{source} if $a^-$ is empty, and is a \emph{sink} if $a^+$ is empty.
A vertex $a\in V$ is said to be a \textit{dominating} vertex (or a \emph{total source}) if $a^+= V\setminus \{a\}$.  The dual notion is that of a \emph{dominated vertex} or \emph{total sink}.

In this section we give some basic observations relating to the computational complexity of CSPs over digraphs formed by taking one-point extensions of digraphs and by taking direct products and disjoint unions.

\begin{proposition}\label{first-order-equivalent}
  Let $G$ be a digraph. Then $\CSP(G^\bot)$ and $\CSP(G^\top)$ are
  equivalent to $\CSP(G)$ under first order reductions.
\end{proposition}

\begin{proof}
  It suffices to prove the result for $G^\top$. Let $H=(V,E)$ be an
  instance of $\CSP(G^\top)$. Let $S$ be the collection of sinks in
  $H$. Let $H'$ be the induced subdigraph of $H$ on the set
  $V\setminus S$. Clearly $H'$ maps homomorphically to $G$ if and
  only if $H$ maps homomorphically to $G^\top$. Observe that
  $V\setminus S=\{ v\in V \mid (\exists u \in V) (v,u)\in E\}$ and
  therefore $\CSP(G^\top)$ is first order reducible to $\CSP(G)$ (a unary first order reduction, with no parameters required).

  Let $H=(V,E)$ be an instance of $\CSP(G)$. Define $H^\top=(V^\top,E^\top)$ as follows:
  \[ 
\text{ For some distinct $a,b\in
    V$ let }
  V^\top=\{ (v,v) \mid v\in V\} \cup \{ (a,b)\}, 
\]
\[
  E^\top=\{ ((u,u),(v,v))\mid (u,v)\in E\} \cup \{ ((v,v),(a,b))\mid
  v\in V \}.
\]
These sets are first order definable in $H=(V,E)$ together with a
linear order on $V$ (let $a$ be the $<$-smallest element and $b$ the $<$-largest), or by letting $a$ and $b$ be the two distinct parameters. Note that $H$ maps homomorphically to $G$ if and
only if $H^\top$ maps homomorphically to $G$, completing the proof.
\end{proof}

We now observe some routine facts about CSPs and complexity.  They are either trivial or folklore, but we give details for completeness.  These facts also hold for general relational structures, though we phrase them here for digraphs.
\begin{lem}\label{lem:CSPplus}
Let $G=(V_1,E_1)$ and $H=(V_2,E_2)$ be a pair of digraphs.
\begin{enumerate}
\item $\CSP(G\times H)=\CSP(G)\cap\CSP(H)$.
\item $\mathbb{K}\in \CSP(G\dotcup H)$ if and only if each component of $\mathbb{K}$ is in at least one of $ \CSP(G)$ or $\CSP(H)$.
\item $\mathbb{K}\in \CSP(G\structcup H)$ if and only if no component of $\mathbb{K}$ has both a $V_1$-related and an $V_2$-related vertex, and each component of $\mathbb{K}$ with neither $V_1$- nor $V_2$-related vertices is in at least one of $ \CSP(G)$ or $\CSP(H)$, each component with a $V_1$-related vertex is in $\CSP(G)$ and each component with an $V_2$-related vertex is in $\CSP(H)$.
\item The problems $\CSP(G)$ and $\CSP(H)$ are first-order reducible to $\CSP(G\structcup H)$.
\end{enumerate}
\end{lem}
\begin{proof}
For the first item, using projections, $G\times H$ maps homomorphically to $G$ and to $H$ so an instance $K$ maps homomorphically to $G\times H$
  then it maps homomorphically to $G$ and $H$.  Conversely, let
  $\varphi_1:K\rightarrow G$ and $\varphi_2:K\rightarrow H$ be
  homomorphisms. Then $\varphi:K\rightarrow G\times H:v\mapsto
  (\varphi_1(v),\varphi_2(v))$ is a homomorphism from $K$ into $G\times H$ as required.  
  
  The second and third cases are essentially trivial and we omit the proofs.  The fourth item is also essentially trivial.  For an instance of $\CSP(G)$, simply add the constraint that all vertices are $G$-related to produce an equivalent instance of $\CSP(G\structcup H)$.
\end{proof}
Item 4 is not in general true if $\structcup$ is replaced by $\dotcup$ or by $\times$.  For instance, for general relational structures, the direct product of two \NP-complete structures may even produce a trivial CSP. \ The following result can easily be extended to general relational structures (using obvious notions of weakly connected component in the proof).

\begin{lemma}\label{poly-constructions}
  Let $G$ and $H$ be digraphs. If $\CSP(G)$ and $\CSP(H)$ both lie in one of the classes $\Poly$, $\NL$, $\Ll$ then  $\CSP(G\times H)$,
  $\CSP(G\dotcup H)$ and   $\CSP(G\structcup H)$ lie in this same class.  \end{lemma}
\begin{proof}
Each of the classes $\Poly$, $\Ll$ and $\NL$ are closed under taking intersections of languages, so if both $\CSP(G)$ and $\CSP(H)$ are contained within one of these complexity classes, then so is $\CSP(G\times H)$ by Lemma~\ref{lem:CSPplus} part~1.

If $\CSP(G)$ and $\CSP(H)$ are in $\Poly$, then Lemma \ref{lem:CSPplus} easily shows that $\CSP(G\dotcup H)$ and $\CSP(G\structcup H)$ are too.  In the case of $\NL$ and $\Ll$ slightly more care must be taken.  We consider $\CSP(G\dotcup H)$ first.

Given any directed graph $K=(V,E)$ and vertex $u\in V$, there is a logspace computation that outputs the vertices of the weakly connected component of $K$ containing $u$: simply treat edges as if they were undirected, and then use the logspace solvability of undirected graph reachability to test which vertices can be reached from $u$ by an oriented path.  

Now if $S$ and $T$ are logspace (or nondeterministic logspace) Turing machine programs for $\CSP(G)$ and $\CSP(H)$ respectively, then for each vertex $u\in V$, we compose the above construction of the component containing $u$ with $S$ and then $T$, looking for acceptance in at least one case.  

The case of $G\structcup H$ is basically the same argument except that there are small checks required to test for the technical conditions relating to no component having both $G$- and $H$-related points and so on.  
\end{proof}

Here are analogous statements for the complexity classes $\Mod_k\Ll$.
\begin{lem}\label{lem:modp}
 Let $G$ and $H$ be digraphs with $\CSP(G)\in \Mod_m\Ll$ and $\CSP(H)\in \Mod_n\Ll$.  If $n=m$ is a prime then $\CSP(G\dotcup H)$, $\CSP(G\structcup H)$ and $\CSP(G\times H)$ are in $\Mod_{n}\Ll$.   In general $\CSP(G\dotcup H)$, $\CSP(G\structcup H)$ are solvable in $\Mod_{nm}\Ll$.
 \end{lem}
 \begin{proof}
These are immediate consequences of Lemma \ref{lem:CSPplus} and the closure properties for $\Mod_p\Ll$ languages: the only technicality is that because neither $\CSP(G\dotcup H)$ nor $\CSP(G\structcup H)$ correspond exactly to $\CSP(G)\cup \CSP(H)$ we cannot directly use the closure of $\Mod_n\Ll$ under unions.  However a component testing argument as in the proof of Lemma \ref{poly-constructions} places both problems in $\Ll^{\Mod_n\Ll}=\Mod_n\Ll$.
  \end{proof}

\section{Polymorphisms and simple constructions}\label{sec:unionspoly}
In this section we observe the preservation of polymorphism properties across forms of union and direct product.  We also observe a natural extension of polymorphisms from a digraph $G$ to a one-point extension $G^\top$ or $G^\bot$.  First we note the following.


\begin{lem}\label{Taylor-const}
  Let $G=(V_1,E_1)$ and $H=(V_2,E_2)$ be digraphs with weak NU polymorphisms $t_1$ and
  $t_2$ of the same arity $n$. Then 
  \[ w(x_1,x_2,\dots,x_n)=
     \begin{cases}
       t_1(x_1,x_2,\dots,x_n) \text{ if $x_1,x_2,\dots,x_n\in V_1$,}\\
       t_2(x_1,x_2,\dots,x_n) \text{ if $x_1,x_2,\dots,x_n\in V_2$,}\\
       x_i \text{ where $i$ is minimal with respect to $x_i\in
         V_1$}
     \end{cases}
     \]
     is a weak NU polymorphism on $G\mathbin{\dot\cup}H$ and $G\mathbin{\overline{\cup}}H$.
\end{lem}
\begin{proof}
  We omit the proof, which is completely routine.
\end{proof}

\begin{lemma}\label{n-perm-union}
  Let $G=(V_1,E_1)$ and $H=(V_2,E_2)$ be digraphs.  If both $G$ and $H$ satisfy one of the following properties then so also do $G\dotcup H$ and $G\structcup H$\up:
\begin{enumerate}
\item congruence $n$-permutability for some $n$\up;
\item congruence modularity\up;
\item the Hobby-McKenzie property.
\end{enumerate}
\end{lemma}
\begin{proof}
Let $i,j$ be such that $\{i,j\}=\{1,2\}$.  If 
\[
(a,b,c)\in (V_j\times V_i\times V_i)\cup (V_i\times V_j\times V_i)\cup (V_i\times V_i\times V_j)
\]
 then the element of $a,b,c$ that lies in $V_j$ is said to be the \emph{minority selection}.

In each case we extend existing ternary polymorphisms $p^G(x,y,z)$ on $G$ and $p^H(x,y,z)$ on $H$ to a single polymorphism $p$ on the union $V_1\cup V_2$ by using one of the following methods.  In all cases, when $\{a,b,c\}\subseteq V_1$ we define $p(a,b,c)=p^G(a,b,c)$ and when $\{a,b,c\}\subseteq V_2$ we define $p(a,b,c)=p^H(a,b,c)$.  When  $\{a,b,c\}\not\subseteq V_1$ and $\{a,b,c\}\not\subseteq V_2$ we use \emph{one} of the following cases:
\begin{enumerate}
\item $p(a,b,c):=a$;
\item $p(a,b,c):=c$;
\item $p(a,b,c)$ is the minority selection from $(a,b,c)$.
\end{enumerate}
It is trivial to verify that (regardless of which of the three choices is made) $p$ will be a polymorphism of both $G\dotcup H$ and $G\structcup H$, provided $p^G$ and $p^H$ were polymorphisms of $G$ and $H$ respectively.  We now observe how to use these extension methods to verify the three conditions in the lemma from $G$ and $H$ to $G\dotcup H$ and $G\structcup H$.

  For congruence permutability, assume that $G$ has polymorphisms witnessing congruence $m$-permutability and $H$ has polymorphisms witnessing congruence $n$-permutability.  Without loss of
  generality assume $n>m$. It is easy to see that $G$ is also
  $n$-permutable by simply extending the chain of polymorphisms by
  third projections. Let $p_1^G,\dots,p_n^G$ and $p_1^H,\dots,p_n^H$ be
  polymorphisms witnessing congruence $n$-permutability on $G$ and $H$
  respectively. 
  We define $p_1$ from $p_1^G$ and $p_1^H$ according to construction $3$ above.  For $i=2,\dots,n$ we use construction 2.  Recall that on tuples from within $V_1$, the polymorphisms $p_1,\dots,p_n$ coincide with $p_1^G,\dots,p_n^G$, while on tuples from within $V_2$, the polymorphisms $p_1,\dots,p_n$ coincide with $p_1^H,\dots,p_n^H$.  For tuples including elements from both $V_1$ and $V_2$, $p_1$ is a minority, while $p_2,\dots,p_n$ are third projections.  It follows that the required equations for congruence $n$-permutability hold.
  
  For congruence modularity, the idea is very similar.  We again assume that the length of the chain of equalities in the Gumm terms (see Subsection \ref{subsec:CM}) is the same on $G$ and $H$, by padding the start using first projections.  Then these terms are extended to $V_1\cup V_2$ by the third choice of method of definition for the term $p$ and the first choice of method for the $s_i$.
  
  For Hobby-McKenzie, we use choice 1 on the terms $d_i$, choice 2 on the terms $e_i$ and choice 3 on the term $p$.
\end{proof}
\begin{thm}\label{thm:sumclosure}
All conditions listed in Figure \ref{fig:maltsev} are preserved under the $\dotcup$ and $\structcup$ constructions\up: if any one is held by both $G$ and $H$, then it is also held by $G\dotcup H$ and $G\structcup H$.  Moreover, if $G$ or $H$ fails one some condition in Figure \ref{fig:maltsev}, then $G\structcup H$ also fails the condition.
\end{thm}
\begin{proof}
The final claim in the theorem holds because polymorphisms of $G\structcup H$ must preserve $G$ and also $H$.  So $G\structcup H$ cannot have have stronger polymorphism properties than $G$ or $H$.

For preservation, note that every condition in Figure \ref{fig:maltsev} is formed by the simultaneous satisfaction of some combination of Taylor, $\SD(\wedge)$, Hobby-McKenzie, congruence $n$-permutability, congruence modularity (in some cases for particular $n$: for example, the Maltsev property is just congruence $2$-permutability).  Thus it suffices to verify that these are preserved.  All except the Taylor and $\SD(\wedge)$ are covered directly by Lemma \ref{n-perm-union}.  

For the $\SD(\wedge)$ property, use the fact that $G$ and $H$ both have $3$-ary and $4$-ary weak NU polymorphisms $w_1$ and $w_2$, and that the construction in Lemma \ref{Taylor-const} preserves the required connecting equations $w_1(y,x,x)=w_2(y,x,x,x)$.

 For the Taylor property, we use a result of Barto and Kozik~\cite{b-k3}, which shows that both $G$ and $H$
   have weak NU polymorphisms of every arity $p$ where $p$
  is a prime greater than the size of the vertex set. So we may choose some
  prime $p$ larger than $|V_1|$ and $|V_2|$ and let $t_1$ and $t_2$ be
  the corresponding weak NU polymorphisms of arity $p$ on $G$ and $H$
  respectively.   Then the construction in Lemma \ref{Taylor-const} gives the desired result.
\end{proof}
\begin{thm}\label{thm:productclosure}
For any of the polymorphism properties $P$ described in Figure \ref{fig:maltsev}, if $G$ and $H$ have $P$ then so does $G\times H$.
\end{thm}
\begin{proof}
This is virtually trivial and we give only a sketch.
If $s$ is a polymorphism on $G$ and $t$ is a polymorphism on $H$, then $s\times t$ (defined as $s$ on the first coordinate and $t$ on the second) is obviously a polymorphism on $G\times H$.  Almost all of the conditions in Figure \ref{fig:maltsev} involve ternary polymorphisms, and the systems of equations can be extended arbitrarily in length using projections.  So, systems of equations witnessing property $P$ on $G$ and property $P$ on $H$ can be combined by $\times$ to witness property $P$ on $G\times H$.  For the case of Taylor, we must find  weak NU polymorphisms of the same arity, which as in the proof of Theorem \ref{thm:sumclosure}, is guaranteed by \cite{b-k3}.  The case of $\SD(\wedge)$ is determined (in our presentation) by polymorphisms of matching arities, so again translates simply by applying $\times$.
\end{proof}
This result also applies to the 2-semilattice property and the property of admitting totally symmetric polymorphisms of all arities.

The following useful lemma is a special case of polymorphisms
preserving primitive positive definable sets and is easy to prove.
\begin{lemma}\label{subalgebra}
  Let $G=(V,E)$ be a digraph with $n$-ary idempotent polymorphism
  $p$. Then for all $v\in V$, the sets $v^+$ and $v^-$ are closed under $p$\up: application of $p$ to any tuple from $v^+$ \up(or $v^-$\up) returns a vertex in $v^+$ \up(or $v^-$\up, respectively\up).
\end{lemma}

\begin{defn}\label{defn:top}
Let $G=(V,E)$ be a digraph and $p$ a polymorphism of arity $n$ on $G$.  We let $p^\top$ denote the operation on $G^\top$ defined by 
\[
p^\top(a_1,a_2,\dots,a_n)=
\begin{cases}
  p(a_1,a_2,\dots,a_n) \text{ if $a_1,\dots,a_n\in V$}\\
 \top \text{ otherwise.}
\end{cases}
\]
The operation $p^\bot$ is defined similarly on $G^\bot$.
\end{defn}
We will write  $p^{\bot\top}$ to denote $(p^\bot)^\top$ and so on.
The following lemma has trivial proof.
\begin{lem}\label{lem:polyplus}
If $p$ is a polymorphism of $G$ then $p^\top$ is a polymorphism of $G^\top$ and $p^\bot$ is a polymorphism of $G^\bot$.
\end{lem}
A \textit{regular} equation is an equation of the form
$s(x_1,x_2,\dots,x_m)= t(y_1,y_2,\dots,y_n)$ where the variables
$\{x_1,\dots,x_m\}=\{y_1,\dots,y_n\}$.  Examples of regular systems of equations are those defining the weak NU (omitting unary type) and omitting semilattice and affine type ($\SD(\wedge)$), as well as any family of totally symmetric idempotent polymorphisms of all arities and the $2$-semilattice equations.

\begin{pro}\label{pro:regular}
  Let $G$ be a digraph.  If a property $P$ is defined by regular polymorphism equations, then $G^{[i,j]}$ has property $P$ if and only if $G$ has property~$P$.  
\end{pro}
\begin{proof}
The backward implication follows by repeated application of Lemma \ref{lem:polyplus}, along with a trivial check that if $s$ and $t$ are polymorphisms satisfying a regular equation, then $s^\top$ and $t^\top$ satisfy the same equation.  The forward implication follows from Lemma \ref{subalgebra}.
\end{proof}
The polymorphism $p^{\top\bot}$ is not identical to $p^{\bot\top}$, even if $(G^\top)^\bot$ is identical to $(G^\bot)^\top$.  So the extension of polymorphisms for $P$ on $G$ to $G^{[i,j]}$ is never unique.

\begin{cor}\label{cor:regular}
Let $G$ be a digraph.  Then the following properties hold on $G$ and $G^{[i,j]}$ equivalently\up: the Taylor property\up; the $\SD(\wedge)$ property\up; the $2$-semilattice property\up; totally symmetric idempotent polymorphisms of all arities.
\end{cor}
\begin{eg}\label{eg:cycle2sl}
Recall that $\mathbb{C}_n$ denotes the directed cycle on $n$ points \up(with no edges if $n=1$\up).  If $i\leq 0\leq j$ then $\mathbb{C}_n^{[i,j]}$ has polymorphisms witnessing $\SD(\wedge)$ and has a $2$-semilattice polymorphism if and only if $n$ is even.  If $n>1$ then $\mathbb{C}_n^{[i,j]}$ does not have a totally symmetric idempotent polymorphisms of arity $n$.
\end{eg}
\begin{proof}
This follows from Example \ref{eg:cycle} and Corollary \ref{cor:regular}.
\end{proof}

\section{Congruence
  $n$-permutability}\label{sec:npermtopbot}
\noindent Consider the following property of digraphs:
\begin{itemize}
\item[] $v^+\subseteq w^+$ and $v^-\subseteq w^-$ implies $v=w$, for any $v,w\in V$.
\end{itemize}
A directed graph failing this property is often called \emph{dismantlable}; otherwise the digraph will be said to be \emph{nondismantlable}.

\begin{lemma}\label{P-for-cores}
If $G$ is a core, then $G$ is nondismantlable. 
\end{lemma}

\begin{proof}
  Let $w,v\in G$ have $v^+\subseteq w^+$ and $v^-\subseteq w^-$.  If
  $w\neq v$, then removing $v$ from $G$ results in a subgraph equal to
  the homomorphic image of $G$ obtained by sending $v$ to $w$. This
  gives a nontrivial retraction of $G$, contradicting $G$ being a
  core. Thus $w=v$ as required.
\end{proof}

Notice that a graph is nondismantlable if and only if its one-point extension (by a total source or by a total sink) is nondismantlable.

\begin{lemma}\label{proj}
  Let $G=(V,E)$ be a digraph and let $p(x,y,z)$ be a polymorphism on
  $G^{\bot\top}$ such that $p(x,y,y) = x$. Then,
\begin{enumerate}
\item $a^+\subseteq p(a,b,c)^+$ for all $a,b,c\in V\cup\{\bot\}$.
\item $a^-\subseteq p(a,b,c)^-$ for all $a,b,c\in V\cup\{\top\}$.
\item if $G$ is nondismantlable, then $p(a,b,c) = a$ for all $a,b,c\in V$.
\end{enumerate}
\end{lemma}
\begin{proof}
  Take $a,b,c\in V\cup\{\bot\}$. By definition we have $b\ra\top$ and
  $c\ra\top$ in $G^{\bot\top}$, and therefore $p(a,b,c)\ra
  p(x,\top,\top) = x$, for any $x$ such that $a\ra x$ in
  $G^{\bot\top}$.  Putting $d:=p(a,b,c)$, we obtain $d\ra x$, for
  any $x\in V$ such that $a\ra x$. Thus $a^+\subseteq d^+$, proving
  (1). The proof of (2) is symmetric.  Finally, for (3), if $a,b,c\in
  V$, then $a^+\subseteq d^+$ and $a^-\subseteq d^-$ both hold, and
  then if $G$ is nondismantlable, we conclude $a=d$.
\end{proof}

Recall that \emph{$1$-permutable} is interpreted to mean that the equation $x=y$ (which is satisfied only by one-element structures) holds.
\begin{theorem}\label{n-perm}
Let $G$ be a digraph and $n\geq 1$.  If $G$ is $n$-permutable then $G^{\bot\top}$ is $(n+2)$-permutable.  If $G$ is nondismantlable
and $G^{\bot\top}$ is $(n+2)$-permutable, then $G$ is congruence $n$-permutable.
\end{theorem}
\begin{proof}
Suppose $G$ is $n$-permutable as witnessed by ternary polymorphisms $p_0,\dots,p_{n}$. We define terms $q_0,\dots, q_{n+2}$ as follows:
\begin{align*}
q_0(x,y,z)&=x\\
q_1(x,y,z) &=\begin{cases}
x & \text{if } y=z\\
p_0^{\top\bot}(x,y,z) & \text{otherwise} 
\end{cases}\\
q_{i+1}(x,y,z) &=p_i^{\top\bot}(x,y,z)\\
q_{n+1}(x,y,z) &=\begin{cases}
z & \text{if } x=y\\
p_n^{\top\bot}(x,y,z) & \text{otherwise}
\end{cases}\\
q_{n+2}(x,y,z)&=z\\
\end{align*}
where $1\leq i\leq n-1$.  Observe also, that the
``otherwise'' case always implies $\{\bot,\top\}\cap\{x,y,z\}\neq\varnothing$.

For $2\leq j\leq n$, the operations $q_j(x,y,z)$ are polymorphisms by Lemma \ref{lem:polyplus}.  For $q_1$ and $q_{n+1}$, the definitions have two cases, but as one case is a projection, and the other is defined in terms of polymorphisms $p$, the property of being a polymorphism could only fail at a pair of tuples $(a,b,c)\rightarrow (a',b',c')$ where one tuple falls into the first case, and the other into the second.  We consider $q_1$, with $q_{n+1}$ very similar.  Assume first that $(a,b,c)\rightarrow (a',b',c')$ with $b=c$ but $b'\neq c'$; in this case $q_1(a,b,c)=a$.  Now $p_0^{\top\bot}$ is also a first projection (so that $a\rightarrow a'=q_1(a',b',c')$ as required) unless $\top$ or $\bot$ is contained in $\{a',b',c'\}$.  But $(a,b,c)\rightarrow (a',b',c')$ ensures that $\bot\notin\{a',b',c'\}$ so that if   $\top$ or $\bot$ is contained in $\{a',b',c'\}$, then $\top\in \{a',b',c'\}$ and $a\rightarrow \top=q_1(a',b',c')$, as required.
The case where $b\neq c$ but $b'=c'$ is an almost identical argument using $\bot$ in place of $\top$.  Thus in every case, adjacency is preserved, and $q_1$ (and by symmetry, $q_{n+1}$) is a polymorphism.

It remains to show that $q_0,\dots,q_{n+2}$ witness $n+2$-permutability 
on $G^{\bot\top}$, that is, $q_0(x,y,z)=x = q_1(x,y,y)$, 
$q_j(x,x,y) = q_{j+1}(x,y,y)$ and $q_{n+1}(x,x,y) = y=q_{n+2}(x,z,y)$ hold for 
all $j\in\{1,\dots, n\}$. Since the first and last pairs of equalities hold by the
respective definitions, we only need to verify the middle one.  Moreover, Lemma \ref{lem:polyplus} shows that all the other equalities hold, with the possible exception of $q_1(x,x,y)=q_2(x,y,y)$ and $q_n(x,x,y)=q_{n+1}(x,x,y)$.  However, if $\{x,y\}\not\subseteq V$, then these again follow from Lemma \ref{lem:polyplus}. If $\{x,y\}\subseteq V$, then $q_1(x,x,y)=x$, while $q_2(x,y,y)=p_1(x,y,y)=x$ also.  Similarly $q_n(x,x,y)=y=q_{n+1}(x,y,y)$ in this case.

Now, to prove the converse for nondismantlable digraphs, suppose $G^{\bot\top}$ is $(n+2)$-permutable, and $G$ is nondismantlable. Let 
$q_0,\dots,q_{n+2}$ be ternary terms witnessing congruence $(n+2)$-permutability. By Lemma~\ref{proj},
$q_1$ and $q_{n+1}$ are respectively the first and third projections on $G$.
In particular, for any $a,b\in V$ we have 
$a = q_1(a,a,b) = q_2(a,b,b)$, where the
second equality follows from the fact that 
$q_1(x,x,y) = q_2(x,y,y)$ holds for any $x,y\in V\cup \{ \bot,\top\}$, which in turn 
follows from $n+2$-permutability of $G^{\bot\top}$ witnessed by 
$q_0,q_1,\dots,q_{n+1},q_{n+2}$. Similarly, for any $a,b\in V$ we have 
$q_{n}(a,a,b) = q_{n+1}(a,b,b) = b$. Now, for $i\in\{0,\dots, n\}$, 
define $p_i$ to be $q_{i+1}$ restricted to $V$.
Each $p_i$ is then a polymorphism on $G$, and 
$p_0,\dots,p_{n}$ satisfy the conditions for $n$-permutability.
\end{proof}

Notice that Theorem~\ref{n-perm} does not apply directly to cases
where $G^{\bot\top}$ is 2-permutable. All we can say in this
case is: if $G^{\bot\top}$ is 2-permutable, then it is also
4-permutable, hence $G$ has a Maltsev term.  

\begin{eg}\label{eg:transtournament}
Let $\mathbb{T}_n$ be the transitive tournament on $n\geq 2$ vertices. Then
$\mathbb{T}_n$ is congruence $n$-permutable, but not congruence $(n-1)$-permutable.  Also, $\mathbb{T}_n$ has a majority polymorphism and totally symmetric idempotent polymorphisms of all arities.
\end{eg}
\begin{proof}
We first consider the congruence $n$-permutability claims.
Clearly $\mathbb{T}_2$ has a Maltsev polymorphism but fails $x=y$.  Also $\mathbb{T}_3$ is $3$-permutable by Theorem \ref{n-perm}.  However $\mathbb{T}_3$ does not have a Maltsev polymorphism because
  every ternary polymorphism $p$ must satisfy $p(1,1,2)\ra p(2,3,3)$
  in $T_3$ and a Maltsev polymorphism would require $2\ra 2$ which is
  not the case.
  The result now follows by an easy induction argument using Theorem \ref{n-perm}  and the base cases $\mathbb{T}_2$ and $\mathbb{T}_3$.
  
To define a majority polymorphism $m(x,y,z)$ on $\mathbb{T}_n$, let $m(x,y,z)$ take the middle value of $x,y,z$ (or majority if $|\{x,y,z\}|\leq 2$).  The operation $t_n(x_1,\dots,x_n):=\min\{x_1,\dots,x_n\}$ is a totally symmetric idempotent polymorphism.
\end{proof}
Relative to Figure \ref{fig:maltsev}, this example shows that the problem $\CSP(\mathbb{T}_k)$ lies at the node labelled ``CD and C$n$P'', where (by Theorem \ref{eg:transtournament}) the precise value of $n$ is $k$.
%
%
%

The following theorem is the Hobby-McKenzie analogue of Theorem \ref{n-perm}.
\begin{thm}\label{thm:HobMcK}
If $G=(V,E)$ is a digraph, then $G^{\top\bot}$ has the Hobby-McKenzie property if and only if $G$ has the Hobby-McKenzie property.
\end{thm}
\begin{proof}
The proof is similar to that of Theorem \ref{n-perm}.  Assume that $G$ has ternary polymorphisms $d_0,\dots,d_n,p,e_n,\dots,e_0$ witnessing the equations for the Hobby-McKenzie property.  
We now define a sequence of ternary polymorphisms $D_0$, \dots, $D_{n+2}$, $P$, $E_{0}$, \dots, $E_{n+2}$ on $G^{\top\bot}$ also witnessing this property.
For $2\leq i\leq n+2$ we define $D_i(x,y,z):=d_{i-2}^{\top\bot}(x,y,z)$ and $E_{i-2}(x,y,z):=e_{i-2}^{\top\bot}(x,y,z)$.  We also let $P(x,y,z)=p^{\top\bot}(x,y,z)$.  Next we define $D_0$ and $E_{n+2}$ to be projections, $D_0(x,y,z):=x$ and $E_{n+2}(x,y,z):=z$.  We define
\[
D_1(x,y,z):=\begin{cases}
d_0^{\top\bot}(x,y,z)&\text{ if }x=y\text{ or }x=z\\
x&\text{ otherwise.}
\end{cases}
\]
If $n$ is even, then we require  $E_{n}(x,y,y)=E_{n+1}(x,y,y)$ and $E_n(x,y,x)=E_{n+1}(x,y,x)$, while $E_{n+1}(x,x,y)=y$.  As $E_n(x,y,y)=e_n^{\top\bot}(x,y,y)$ we require
\[
\text{ (for $n$ even)}\qquad E_{n+1}(x,y,z):=\begin{cases}
z&\text{ if }x=y\\
e_n^{\top\bot}(x,y,z)&\text{ otherwise.}
\end{cases}
\]
If $n$ is odd, then we require  $E_{n}(x,x,y)=E_{n+1}(x,x,y)$ with 
$E_{n+1}(x,y,x)=E_{n+2}(x,y,x)=x$ and $E_{n+1}(x,y,y)=y$.  In this case we define
\[
\text{ (for $n$ odd)}\qquad E_{n+1}(x,y,z):=\begin{cases}
z&\text{ if }x=z\text{ or }y=z\\
e_n^{\top\bot}(x,y,z)&\text{ otherwise.}
\end{cases}
\]
We now verify that these are polymorphisms.  Consider a pair of adjacent tuples $(a,b,c)\rightarrow (a',b',c')$.  
The verification that $D_1$ is a polymorphism is essentially identical to the argument in Theorem \ref{n-perm}; we omit further details.

For $E_{n+1}$ we have two cases.  Let $n$ be even.  As $e_n(x,y,z)=z$ always, 
the two possible cases in the definition of $E_{n+1}(x,y,z)$ agree (and $E_{n+1}(x,y,z)=z$) unless $\top$ or $\bot$ appear in $\{x,y,z\}$.  Now if $\top$ or $\bot$ appear in  $(a,b,c)$ then it can only be $\bot$, while if $\top$ or $\bot$ appear in $(a',b',c')$ it can only be $\top$.  In the first case we have that $E_{n+1}(a',b',c')\in\{c',\top\}$, while $E_{n+1}(a,b,c)\in \{c,\bot\}$.  In each of the four possible cases, adjacency is preserved.  The case where $\top$ appears in $(a',b',c')$ is very similar.

Now let $n$ be odd.  In this case it remains true that the two cases defining $E_{n+1}$ are in agreement unless $e_n^{\top\bot}$ fails to act as a third projection, which is if and only if $\top$ or $\bot$ is contained in $\{x,y,z\}$.  In the case of the adjacency $(a,b,c)\rightarrow (a',b',c')$, it is now seen that the previous argument holds without change.

The verification that $D_0,\dots,D_{n+2},P,E_0,\dots,E_{n+2}$ satisfy the equations required to witness the Hobby-McKenzie property is routine: the equations are either regular, or we have defined them precisely in terms of the required equations.
\end{proof}
We mention that one can prove a kind of converse in the case of nondismantlable digraphs, showing that the length of the chain of equations determining the Hobby-McKenzie property \emph{must} increase under the addition of two sided extensions, however unlike the situation for congruence $n$-permutability, the length of these equations does not tie directly to a natural algebraic property, so we do not pursue this argument.
\begin{cor}
If $G$ is a digraph, then $G$ has polymorphisms witnessing the $\SD(\vee)$ property if and only if $G^{[i,j]}$ has.
\end{cor}
\begin{proof}
This is because the $\SD(\vee)$ condition is equivalent to the simultaneous satisfaction of the $\SD(\wedge)$ and Hobby-McKenzie properties.  These properties are stable under the addition of one-point extensions by Corollary \ref{cor:regular} and Theorem \ref{thm:HobMcK}.
\end{proof}
\begin{eg}\label{eg:cycleSDjoin}
Let $\mathbb{C}_n$ be the directed $n$-cycle, with $n>1$.  Then for $i\leq 0\leq j$ we have $\mathbb{C}_n^{[i,j]}$ satisfying $\SD(\vee)$ and congruence $(2\max\{|i|,|j|\}+2)$-permutability, but not congruence $2\min\{|i|,|j|\}$-permutability.
\end{eg}
\begin{proof}
Without loss of generality assume $|i|\leq |j|$.  Recalling Example \ref{eg:cycle} , we have that $\mathbb{C}_n$ is congruence $2$-permutable and so $\mathbb{C}_n^{[-j,j]}$ is congruence $(2j+2)$-permutable by Theorem \ref{n-perm}.  But as $\mathbb{C}_n$ is nondismantlable, Theorem \ref{n-perm} also shows that $\mathbb{C}_n^{[i,-i]}$ is not congruence $(2i)$-permutable.  (Strictly this involves an easy induction, starting from the fact that $\mathbb{C}_n$ is not trivial---that is, not 1-permutable---when $n>1$.)
\end{proof}
Note that if $n=1$ then $\mathbb{C}_1^{[i,j]}$ is simply $\mathbb{T}_{j-i+1}$, and the corresponding facts are covered by Example \ref{eg:transtournament}.

\section{Permutational digraphs and
  $n$-permutability}\label{sec:permutation}
\noindent We call a digraph $G = (V,R)$ \emph{permutational} if it is
a disjoint union of directed cycles, and hence $R$ is a permutation
on $V$. Since we assume that digraphs have no loops, a permutational
digraph $G$ is nontrivial and the permutation defined by $R$ has no
fixpoints.  Observe that any permutational digraph has a Maltsev
polymorphism: simply define $m(x,y,z)=x$ whenever $x\neq y$ or $y\neq
z$ and the Maltsev equations otherwise.

Let $G$ be a digraph. Define the \emph{depth} of a vertex $v$ in
$G^{[0,n]}$, for some $n>0$, as the smallest $k\geq 0$ such
that $v$ belongs to $G^{[0,k]}$. Let $\delta(v)$ denote the depth of~$v$.  
The next lemma spells out some properties of $G^{[0,n]}$ that we
will make use of in this section.

\begin{lemma}\label{Gn-props}
  Let $G=(V,E)$ be permutational and $n\geq 0$.  Then, the following
  hold in $G^{[0,n]}$:
\begin{enumerate}
\item $x^+\subseteq y^+$ implies $\delta(x) \geq \delta(y)$. 
\item $\delta(x) = \delta(y)> 0$ implies $x = y$.
\item $\{x,y\}\not\subseteq V$ implies  there exists $z\in V$ such that
$z\ra x$ and $z\ra y$.
\item $x^-\subseteq y^-$ and $x\neq y$ imply $x\ra y$.
\end{enumerate}
\end{lemma}

\begin{proof}
  From the definition of $G^{[0,n]}$ it is easily seen that (1), (2)
  and (3) hold.  For~(4), observe that if $x^-\subseteq y^-$ and
  $x\neq y$ hold, then $y\notin V$.  If $x\in V$ then then $x\ra y$ holds by
  definition.  If $x\notin V$ then the conditions $x^-\subseteq y^-$ and
  $x\neq y$ (and the fact that $y\notin V$) again gives $x\ra y$.
\end{proof}

\begin{lemma}\label{pre-perm-proj}
  Let $G=(V,E)$ be a permutational digraph. If $p(x,y,z)$ is a polymorphism
  on $G^{[0,n+1]}$ with $p(x,y,y)=x$, then $p(a,b,c)=a$, for all
  $a,b,c\in V\cup \{ 1,2,\dots,n\}$.
\end{lemma}

\begin{proof}
  The proof will be broken into cases; $\{a,b,c\}\subseteq V\cup\{1,2,\dots,n\}$ is assumed throughout.

  \noindent \emph{Case 1.} Assume that $a,b,c\in V$ and let
  $d:=p(a,b,c)$. Since $p(a,b,c)\ra p(1,1,1)=1$ we have $d\in V$. Let $a'\in V$ be the
  unique element such that $a\ra a'$. Then $d=p(a,b,c)\ra
  p(a',1,1)=a'$ and therefore $d=a$.  \bigskip

  \noindent \emph{Case 2.} Assume that $\{ b,c \} \not\subseteq V$ and let $d=p(a,b,c)$.  Therefore $d=p(a,b,c)\ra
  p(a',n+1,n+1)=a'$, for all $a'\in a^+$ and so $a^+\subseteq d^+$.
  By Lemma~\ref{Gn-props}~(3) there is a $z\in V$ such that $z\ra b$
  and $z\ra c$. Therefore $a'=p(a',z,z)\ra p(a,b,c)=d$, for all $a'\in
  a^-$ and so $a^-\subseteq d^-$. As $G^{[0,n+1]}$ is nondismantlable we
  conclude that $a=d$.  \bigskip

  \noindent \emph{Case 3.}  Assume, $a\notin V$ and $b,c\in V$ and let
  $p(a,b,c) = d$. Since $a\notin V$ we have $\delta(a)\geq 1$.  As
  $d=p(a,b,c)\ra p(a',1,1)=a'$, for all $a'\in a^+$ we have
  $a^+\subseteq d^+$.  Applying Lemma~\ref{Gn-props}(1) gives
  $\delta(a)\geq\delta(d)$ and so either $a=d$ or $d\ra a$.  We prove
  by induction on $\delta(a)$ that $p(a,b,c)=a$.

  If $\delta(d) = 0$ and $a\neq d$, then $d\in V$, and by Case 1 we
  get $p(d,v,w) = d$, where $\{v\} = b^-\cap V$ and $\{w\} = c^-\cap
  V$.  It follows that $d = p(a,b,c)\la p(d,v,w) = d$. As $G$ has no
  loops, we obtain a contradiction.

  Now assume for some $k\leq n$ that $p(x,y,z)=x$ whenever $y,z\in V$ and $\delta(x)\leq k$ and consider $\delta(a) = k+1$. If $a\neq d$, then
  $\delta(d) < k+1$ and $d\ra a$. Now, the inductive hypothesis
  applies to $p(d,v,w)$ with $\{v\} = b^-\cap V$ and $\{w\} = c^-\cap
  V$, so $p(d,v,w) = d$. Thus, $d = p(a,b,c)\la p(d,v,w) = d$, a
  contradiction.
\end{proof}

\begin{proposition}\label{n-perm-Gn}
  Let $G$ be permutational and $n\in\mathbb{N}$. Then, $G^{[0,n]}$ is
  $(2n+2)$-permutable but not $(2n+1)$-permutable.
\end{proposition}

\begin{proof}
  The proof will be by induction on $n$. For $n = 0$ the claim follows
  since all permutational digraphs have a Maltsev
  polymorphism. Consider $G^{[0,n+1]}$ and suppose it is
  $(2n+3)$-permutable.  Let $p_1,\dots,p_{2n+2}$ be the polymorphisms
  witnessing that. In particular, $x = p_1(x,y,y)$ and
  $p_{2n+2}(x,x,y) = y$ hold. By Lemma~\ref{pre-perm-proj}, we have that $p_1$ and
  $p_{2n+2}$ are the first and the third projections on
  $G^{[0,n]}$ respectively. It follows that the polymorphisms
  $p_2,\dots,p_{2n+1}$ satisfy the conditions for
  $(2n+1)$-permutability of $G^{[0,n]}$. But $G^{[0,n]}$ is not
  $(2n+1)$-permutable by the inductive hypothesis, a contradiction.
\end{proof}
Proposition \ref{n-perm-Gn} and Theorem \ref{n-perm} combine to show that in general it is \emph{not true} that the length of permutability (that is, the parameter $n$ in $n$-permutability) must increase under one point extensions.
If $G$ is a nontrivial core and permutational (for example, if $G$ is a single directed cycle), then $G$ is congruence $2$-permutable, so that repeated applications of Theorem \ref{n-perm} shows that $G^{[-n,n]}$ is congruence $(2n+2)$-permutable but not $(2n+1)$-permutable.  Proposition \ref{n-perm-Gn} shows that $G^{[0,n]}$ is also congruence  $(2n+2)$-permutable but not $(2n+1)$-permutable, yet $G^{[-n,n]}$ is obtained from $G^{[0,n]}$ by a sequence of $n$ one-point extensions.

We now revisit our recurrent example of the directed cycles, refining part of the statement of Example \ref{eg:cycleSDjoin}.
\begin{eg}\label{eg:cyclenperm}
Let $n>2$ and $i\leq 0\leq j$.  Then $\mathbb{C}_n^{[i,j]}$ is congruence $(2\max\{-i,j\}+2)$-permutable but not $(2\max\{-i,j\}+1)$-permutable.
\end{eg}
\begin{proof}
For $n>2$ the digraph $\mathbb{C}_n$ is permutational, so that $\mathbb{C}_n^{[i,0]}$ is not congruence $(1-2i)$-permutable and $\mathbb{C}_n^{[0,j]}$ is not congruence $(2j+1)$-permutable.  The  congruence $(2\max\{-i,j\}+2)$-permutability is given in Example \ref{eg:cycleSDjoin}.
\end{proof}
\section{Congruence modularity and distributivity}\label{sec:congdist}
So far, all of the important polymorphism properties have been preserved under basic constructions, though in the case of congruence $n$-permutability, the precise length of permutability is in general unstable.  We now show that the remaining properties in Figure \ref{fig:maltsev} are in fact destroyed under one-point extensions. This is of particular interest because of the fact that some of these polymorphism properties have been proven to correspond to solvability by particular kinds of algorithms (thus these too are unstable under first order reductions).  Referring to Figure \ref{fig:maltsev}, we find that if $\CSP(\mathbb{A})$ lies in one of the the eight regions not on the top layer of the diagram, then one point extensions can in general push the corresponding CSP (which is first order equivalent to $\CSP(\mathbb{A})$) upwards to the top layer.  For example, a CSP satisfying polymorphisms for the class A (arithmetical) will in general lie at the node labeled by ``$\SD(\vee)$ and C$n$P'' after applying sufficiently many one-point extensions.
\begin{figure}[h]
\begin{small}
\begin{center}
  \xymatrix{ 
   &&&& *+\txt{$a'$}\ar@{->}[rrd] &&*+\txt{$a$}\ar@{->}[ll]\ar@{->}[d]\\
   &&&&  & & *+\txt{$1$}\ar@{->}[r] & *\txt{}
    \\
    &&&&*+\txt{$b'$}\ar@{->}[rru] &&*+\txt{$b$}\ar@{->}[ll]\ar@{->}[u]\\
}
\end{center}
\end{small}\caption{Disallowing congruence modularity; if $a^+\cap 1^- =\{ a'\}$ and
    $a'^-\cap 1^-=\{ a\}.$}\label{fig:CM}
\end{figure}
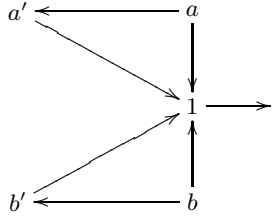

\begin{lemma}\label{no-Gumm}
  Let $G=(V,E)$ be a digraph. Assume that $G$ has
  pairwise distinct vertices $a,b,1$ such that
  \begin{enumerate}
  \item $1^+\neq \varnothing$,
  \item $\exists a'\in V$ such that $a^+\cap 1^- =\{ a'\}$ and
    $a'^-\cap 1^-=\{ a\}$,
  \item $b\in 1^{-}$ and $\exists b'\in 1^{-}$ such that $b'\in b^+$.  {\rm (}See
    Figure~\ref{fig:CM}.{\rm )}
  \end{enumerate}
  Then $G$ is not congruence modular.
\end{lemma}

\begin{proof}
Suppose for contradiction that $s_0(x,y,z),\dots,s_{2n}(x,y,z),p(x,y,z)$ are polymorphisms witnessing the Gumm equations required for congruence modularity condition.

Let $2\in 1^+$.  For $i=0,\dots,2n$ we have $s_i(a',1,b')\ra s_i(1,2,1)=1$, and therefore $s_i(a',1,b')\in 1^{-}$.  Next we show that $s_i(a,a,b)=s_i(a,b,b)=a$ for all $i$.  This is trivially true at $i=0$, so assume we have established it up to some $i\leq 2n$.  We consider the case of $i$ odd, but the case of $i$ even is almost identical.  When $i$ is odd, $a=s_i(a,a,b)=s_{i+1}(a,a,b)$ using the induction hypothesis and the Gumm equalities.  Then $a=s_{i+1}(a,a,b)\ra s_{i+1}(a',1,b')$, so that $s_{i+1}(a',1,b')\in a^{+}\cap 1^{-}=\{a'\}$.  Then $s_{i+1}(a,b,b)\ra s_{i+1}(a',1,b')=a'$ and $s_{i+1}(a,b,b)\ra s_{i+1}(1,1,1)=1$ so that $s_{i+1}(a,b,b)\in a'^{-}\cap 1^{-}=\{a\}$, giving $s_{i+1}(a,b,b)=a$.  Similarly, $s_{i+1}(a,a,b)\ra s_{i+1}(a',1,b')=a'$ and $s_{i+1}(a,a,b)\ra s_{i+1}(1,1,1)=1$, also giving $s_{i+1}(a,a,b)=a$, as required.

We have shown that $a=s_{2n}(a,b,b)$ and then the Gumm equalities show that $a=s_{2n}(a,b,b)=p(a,b,b)$ and $p(a,a,b)=b$.  But $a=p(a,b,b)\ra p(1,1,b')=b'$ so that $b'\in a^{+}\cap 1^{-}=\{a'\}$.  So $a'=b'$.
Then $p(a,a,b)\ra p(1,1,a')=a'$ (using $p(x,x,y)=y$ for the equality) and $p(a,a,b)\ra p(1,1,1)=1$.  So $p(a,a,b)\in a'^{-}\cap 1^{-}=\{a\}$.  But $p(a,a,b)=b$ using equality $p(x,x,y)=y$.  This contradicts $a\neq b$.
\end{proof}
The following example shows that Lemma \ref{no-Gumm} is quite widely applicable.
\begin{eg}\label{eg:CMbreak}
Let $G$ be a digraph containing edges $(u_1,v_1)$ and $(u_2,v_2)$, and with $u_1\neq u_2$.  If $u_1^+=\{v_1\}$ and $v_1^-=\{u_1\}$ then $G^{\top\top}$ is not congruence modular.
\end{eg}
\begin{proof}
Apply Lemma \ref{no-Gumm}: let  $a:=u_1$, $a':=v_1$ and let $b:=u_2$ and $b':=v_2$.
\end{proof}

\begin{eg}\label{eg:cyclemod}
Let $n>1$ and $i\leq 0\leq j$.  If $\max\{-i,j\}>1$ then $\mathbb{C}_n^{[i,j]}$ is not congruence modular.
\end{eg}
\begin{proof}
Without loss of generality, assume that $j>1$. 
 We apply Lemma~\ref{no-Gumm} to $\mathbb{C}_n^{[0,j]}$, for if $i< 0$ and $\mathbb{C}_n^{[i,j]}$ were congruence modular, then using vertex $v:=-1$ we may apply  Lemma \ref{subalgebra} to deduce that $\mathbb{C}_n^{[0,j]}$ is congruence modular.  Let $a$ and~$b$ be two distinct vertices of $\mathbb{C}_n$, and let $a'$ and $b'$ in $\mathbb{C}_n$ be such that $a\ra a'$ and $b\ra b'$.  Lemma \ref{no-Gumm} now applies, using the same notation $a,b\in \mathbb{C}_n$ and 
with the assumption $j>1$ allowing vertex $1$ of $\mathbb{C}_n^{[i,j]}$ to play its stated role in Lemma~\ref{no-Gumm}.
\end{proof}
This enables a precise placement of $\CSP(\mathbb{C}_m^{[i,j]})$ in Figure \ref{fig:maltsev}, at least where $\max\{-i,j\}>1$ and $m>1$ (recall that $m=1$ produces a transitive tournament): there are polymorphisms witnessing $\SD(\vee)$ and C$n$P  (by Example \ref{eg:cycleSDjoin}), but not congruence modularity, so it lies at the node labelled by ``$\SD(\wedge)$ and C$n$P'' (where the precise value of $n$ is $2\max\{-i,j\}+2$, using Example \ref{eg:cyclenperm}).
\begin{eg}\label{eg:cycletopbot}
Let $n>1$.  Then $\mathbb{C}_n^{\top\bot}$ has majority polymorphisms \up(that is, ternary NU polymorphisms\up).
\end{eg}
\begin{proof}
 Recall that $\mathbb{C}_n$ has a majority polymorphism $m$, as defined Example \ref{eg:cycle}.  We now extend this to a majority polymorphism $m'$ on $\mathbb{C}_n^{\top\bot}$.  A \emph{majority configuration} is a 3-tuple $(a,b,c)$ with $|\{a,b,c\}|<3$, and the majority value, denoted $\operatorname{maj}(a,b,c)$ is the value amongst $a,b,c$ that is repeated.
 \[
 m'(a,b,c):=\begin{cases}
 m(a,b,c)&\text{ if }\{a,b,c\}\subseteq C_n,\\
 \operatorname{maj}(a,b,c)&\text{ if }(a,b,c)\text{ is a majority tuple},\\
 \top&\text{ if $|\{a,b,c\}|=3$ and $\{a,b,c\}\not\subseteq C_n$ and $\bot\notin\{a,b,c\}$},\\
 \bot&\text{ if $|\{a,b,c\}|=3$ and $\bot\in \{a,b,c\}$}.
 \end{cases}
 \]
 The majority equations $m(y,x,x)=m(x,y,x)=m(x,x,y)=x$ hold by definition, so it remains to show that $m'$ is a polymorphism.  All cases are basically trivial, except for majority configurations containing $\top$ (or $\bot$, but not both: these are isolated tuples), but where $\top$ (or $\bot$, respectively) is not the majority.  As an example, consider an adjacency $(a,b,c)\ra (\top,d,d)$.  If $m'(a,b,c)=\bot$ or if $(a,b,c)$ is a majority configuration, we are done.  Otherwise, $\{a,b,c\}\subseteq C_n$ and there is $a'\in C_n$ with $(a,b,c)\ra (a',d,d)$.  Then $m'(a,b,c)=m(a,b,c)\ra m(a',d,d)=d=m'(\top,d,d)$.  
\end{proof}
This enables placement of the cases $\mathbb{C}_n^{[0,1]}$, $\mathbb{C}_n^{[-1,0]}$ and $\mathbb{C}_n^{[-1,1]}$ missed by Example~\ref{eg:cyclemod}.  These are all congruence distributive and congruence $4$-permutable, but not congruence $3$-permutable.

Mar\'oti and Z\'adori~\cite{MZ} showed that congruence modularity
implies congruence distributivity (indeed they showed that congruence
modularity implies the existence of a near-unanimity polymorphism) for
reflexive digraphs. The following result gives another class of
digraphs in which congruence modularity implies congruence
distributivity.

\begin{theorem}\label{CM-is-CD}
  Let $G$ be a digraph. The following are equivalent\up:
  \begin{enumerate}
  \item $G^{\top\bot}$ is congruence modular.
  \item $G^{\top\bot}$ is congruence distributive.
  \end{enumerate}
\end{theorem}

\begin{proof}
  Clearly (2) implies (1). For the converse, suppose $G^{\top\bot}$ is congruence
  modular. Then $G^{\top\bot}$ has (ternary) Gumm polymorphisms, say,
  $s_1,\dots,s_{2m},q$.  Define $q'$ on $G^{\top\bot}$  as follows: 
  \[
  q'(x,y,z) =
  \begin{cases}
    q(x,y,z) & \text{ if } x\neq z\\
    x & \text{ if } x = z.
  \end{cases}
  \]
  Clearly $q'$ satisfies the equations $q'(x,x,y)=y$, $q'(x,y,x)=x$
  and $q'(x,y,y)=q(x,y,y)$. We now show that $q'$ is a polymorphism.
  Consider the configuration $x\ra a$, $y\ra b$, $z\ra c$ and suppose
  $q'(x,y,z)\not\ra q'(a,b,c)$. There are only two cases in which this
  could happen. Assume first that $a=c$ and $x\neq z$. Thus,
  $q'(a,b,c) = q'(a,b,a) = a$ and $q'(x,y,z) = q(x,y,z)$, and it is
  also clear that $\top\notin\{x,y,z\}$ and
  $\bot\notin\{a,b\}$. Letting $u: = q(x,y,z)$ and applying
  Lemma~\ref{proj}(1), we obtain $z^+\subseteq u^+$. But, $z\ra c =
  a$, so $a\in z^+$, hence $a\in u^+$.  Therefore, $u\ra a$,
  contradicting $q'(x,y,z)\not\ra q'(a,b,c)$.  For the second case,
  with $a\neq c$ and $x=z$, the proof is analogous, using
  Lemma~\ref{proj}(2).

  By replacing $q$ by $q'$, we obtain polymorphisms
  $s_1,\dots, s_{2m}, q'$, such that $q'(x,y,x) = x$ is satisfied.
  Since the equation $s_{2m}(x,y,y) = q(x,y,y)$ holds, by definition
  of $q'$ we obtain $s_{2m}(x,y,y) = q'(x,y,y)$. 
  This shows that
  $s_1,\dots,s_{2m},q'$ are J\'onsson terms on $G^{\top\bot}$.  (Technically, the Gumm terms must be extended at the start by a further projection term to precisely match the J\'onsson term equations.)
\end{proof}
The following result combines the main result of \cite{BDJN,BDJN2} with Theorem \ref{CM-is-CD} and Theorem \ref{n-perm}.

\begin{thm}\label{thm:collapse}
Fix any $k\in \mathbb{N}$. For every relational structure $\mathbb{A}$ there is a finite digraph $\mathcal{D}_k(\mathbb{A})$ \up(first order definable on a subset of a power $\mathbb{A}$\up) such that the following hold.
\begin{enumerate}
\item $\CSP(\mathbb{A})$ and $\CSP(\mathcal{D}_k(\mathbb{A}))$ are equivalent under logspace reductions.
\item if $\mathcal{D}_k(\mathbb{A})$ has Gumm terms then it has an NU term.
\item if $\mathcal{D}_k(\mathbb{A})$ is congruence $2k$-permutable then it has a majority term.
\end{enumerate}
\end{thm}
\begin{proof}
In \cite{BDJN, BDJN2} it is shown how to construct a digraph $\mathcal{D}(\mathbb{A})$ such that $\CSP(\mathbb{A})$ and $\CSP(\mathcal{D}(\mathbb{A}))$ are equivalent under logspace reductions.  While the $\mathcal{D}(\mathbb{A})$ construction is slightly technical, it is trivially verified to be nondismantlable.  Let $\mathcal{D}_k(\mathbb{A})$ be $\mathcal{D}(\mathbb{A})^{[-k,k]}$.  If $\mathcal{D}_k(\mathbb{A})$ has Gumm terms then by Theorem \ref{CM-is-CD} it has J\'onsson terms.  Then by Barto \cite{bar:NU} it has an NU term.  Now assume that $\mathcal{D}_k(\mathbb{A})$ has terms witnessing $2k$-permutability.  By Theorem \ref{n-perm} it follows that $\mathcal{D}(\mathbb{A})$ is trivial.  Then $\mathcal{D}_k(\mathbb{A})$ is a transitive tournament so has a majority polymorphism (see Example \ref{eg:transtournament}).
\end{proof}
In particular, when $k=1$ the proof of Theorem \ref{thm:collapse} shows that the class of digraphs with a total source and total sink exhibits restricted polymorphism and algorithmic behaviour: for instance, any problem solvable by the few subpowers algorithm is already solvable by local consistency check algorithm.  However every fixed finite template CSP is logspace equivalent to a CSP over a digraph from this class.

\section{Semicomplete digraphs}\label{sec:precomplete}
 As an 
 illustrative consequence of the various results above, we extract a characterisation of the possible computational complexity and polymorphism properties of \emph{semicomplete} digraphs, in the sense of \cite{BJHM}: finite digraphs for which the symmetric closure of the
 edge relation produces a complete graph.  Equivalently, a digraph is semicomplete if for
 every pair of distinct vertices $u,v$, at least one of $(u,v)$ and
 $(v,u)$ is an edge.  Tournaments are perhaps the most commonly encountered instance of a
 semicomplete digraph.  
 
 A simple classification of the tractable CSPs over semicomplete digraphs is given in \cite{BJHM} and has been extended to the broader class of \emph{locally semicomplete digraphs} by Bang-Jensen, MacGillivray and Swarts in \cite{BJMS}.  We now provide a fine level characterisation of complexity classes and polymorphism properties in the semicomplete case; it would be interesting to see this extended to the locally semi-complete digraphs of \cite{BJMS}.

 We make essential use the \emph{no-sources and
   sinks} Theorem of Barto, Kozik and Niven \cite{BKN}, which states
 that the CSP over a core digraph with no sources and sinks is
 \texttt{NP}-complete unless it is a disjoint union of directed cycles
 (in which case it is tractable, and moreover has strict width).

 A semicomplete digraph is always a core, and there can be at most one
 source and at most one sink: indeed a source in a semicomplete
 digraph is a dominating vertex for the entire digraph, and dually for
 a sink.  This means that for a given a semicomplete digraph
 $G=(V,E)$ with a source (or sink) $s$, the set $s^+$ (or $s^-$) is
 equal to $V\setminus \{s\}$.

\begin{thm}
The algebraic dichotomy conjecture holds for semicomplete digraphs.
\begin{enumerate}
\item If a semicomplete digraph $G$ is \emph{not} one of $\mathbb{T}_k$, $\mathbb{C}_2^{[i,j]}$ or $\mathbb{C}_3^{[i,j]}$ for some $i\leq0\leq j$, then $G$ does not have weak NU polymorphisms and $\CSP(G)$ is \NP-complete.
\item $\CSP(\mathbb{T}_k)$ is first order definable, so solvable within \Ll.  The digraph $\mathbb{T}_k$ has a majority polymorphism, as well as a $2$-semilattice polymorphism and is congruence $k$-permutable but not congruence $k-1$ permutable.
\item $\CSP(\mathbb{C}_2^{[i,j]})$ is \Ll-complete.  The digraph $\mathbb{C}_2^{[i,j]}$ is $\SD(\vee)$ and is congruence $(2\max\{-i,j\}+2)$-permutable but not $(2\max\{-i,j\}+1)$-permutable.  It does not have any commutative binary polymorphism.
\item $\CSP(\mathbb{C}_3^{[i,j]})$ is \Ll-complete.  The digraph $\mathbb{C}_3^{[i,j]}$ is $\SD(\vee)$ and is congruence $(2\max\{-i,j\}+2)$-permutable but not $(2\max\{-i,j\}+1)$-permutable.  It has a $2$-semilattice polymorphism.
\item If $\max\{-i,j\}\leq 1$ then $\mathbb{C}_2^{[i,j]}$ and $\mathbb{C}_3^{[i,j]}$ have majority polymorphisms so are congruence distributive.
\item If $\max\{-i,j\}>1$ then $\mathbb{C}_2^{[i,j]}$ and $\mathbb{C}_3^{[i,j]}$ are not congruence modular.
\end{enumerate}
\end{thm}
\begin{proof}
The proof is essentially an application of the main result of \cite{BKN} followed by a summary of various examples considered earlier in this article.

(1) By Proposition \ref{first-order-equivalent}, $\CSP(G)$ is first order equivalent to $\CSP(H)$ for some digraph with no sources and sinks that is obtained from $G$ by deleting total sources and total sinks.  Because $G$ is not $\mathbb{T}_k$, $\mathbb{C}_2^{[i,j]}$ or $\mathbb{C}_3^{[i,j]}$, it follows that $H$ is not one of $\mathbb{C}_1,\mathbb{C}_2,\mathbb{C}_3$, and moreover because $H$ is semicomplete, it is a core without sources and sinks that is not a disjoint union of cycles.  Hence by the main result of \cite{BKN}, $H$ has no weak NU polymorphism so that $\CSP(H)$ is $\NP$-complete.  Hence $\CSP(G)$ is \NP-complete and by Corollary \ref{cor:regular} has no weak NU polymorphism.

  (2) $\CSP(\mathbb{T}_k)$ is well known to be first order definable for all $k\geq 1$ (see \cite{LLT} for example),
  and therefore $\CSP(\mathbb{T}_k)$ is solvable in \texttt{L}. The polymorphism claims are given in Example \ref{eg:transtournament}.

(3--6)  The problems
  $\CSP(\mathbb{C}_2)$ and $\CSP(\mathbb{C}_3)$ are  \texttt{L}-complete.  Indeed, it is well known that these problems are not first order definable (this can be easily proved directly, or otherwise use the classification of first order definable CSPs in Larose, Loten and Tardiff \cite{LLT}); but also as both have majority and Maltsev polymorphisms they are solvable in \Ll\  (by Dalmau and Larose \cite{dallar} for example).  By Larose and Tesson \cite{lartes}, a CSP lying in $\Ll$ but not first order definable is $\Ll$-complete with respect to first order reductions.  Now by
  Lemma~\ref{first-order-equivalent}, $\CSP(\mathbb{C}_2^{[i,j]})$ and
  $\CSP(\mathbb{C}_3^{[i,j]})$ are also \texttt{L}-complete, for all $i\leq 0\leq j$.  The polymorphism claims are established in Example \ref{eg:cyclenperm} for congruence $n$-permutability, Example~\ref{eg:cycletopbot} for majority when $\max\{-i,j\}\leq 1$, Example~\ref{eg:cyclemod} for the failure of congruence modularity when $\max\{-i,j\}> 1$, and Example~\ref{eg:cycle2sl} for the claims about $2$-semilattice polymorphisms.
\end{proof}

\section{Examples separating nodes in Figure \ref{fig:maltsev}}
Kazda \cite{kaz} showed that a digraph with a Maltsev polymorphism also has a majority polymorphism.  This means that Figure \ref{fig:maltsev} undergoes a collapse when restricted to digraph CSPs: the node labelled ``Maltsev'' is identified by that labelled by ``A''.  On the other hand, Bulin, Delic, Jackson and Niven \cite{BDJN2} show that
whenever $\mathbb{A}$ is a finite relational structure with finitely many relational symbols, there is a digraph $\mathbb{D}_\mathbb{A}$ such that \emph{any} other combination of polymorphism properties in Figure \ref{fig:maltsev}, is held equivalently by $\mathbb{A}$ and $\mathbb{D}_\mathbb{A}$.  This does not in itself imply that each node in Figure \ref{fig:maltsev} is genuinely distinct.  

We now use some of the examples in the present article to show that each node of Figure \ref{fig:maltsev} really is distinct.  Using the result of \cite{BDJN2}, this shows that even amongst digraph CSPs each node is distinct, except for the aforementioned ``$\text{Maltsev}\equiv\text{A}$'' collapse.  Using structured unions and Theorem \ref{thm:sumclosure}, it suffices to find examples lying precisely at nodes that are join irreducible in Figure \ref{fig:maltsev}. (We note that while the direct product of relational structures corresponds to CSP intersection, the nodes in Figure \ref{fig:maltsev} correspond to classes of CSPs, and the intersection of two actual CSPs usually lies higher \emph{up} in Figure \ref{fig:maltsev}, not down.)  We list these nodes, and relational structures placing them precisely at these locations.
\begin{itemize}
\item $\SD(\wedge)$.  Recall that the graph of a binary operation $\cdot$ is the ternary relation $\{(x,y,z)\mid x\cdot y=z\}$.  The structure on $\{0,1\}$ with the ternary relation $\{(0,0,0),(0,1,0),(1,0,0),(1,1,1)\}$ corresponding to the graph of the meet semilattice relation on $\{0,1\}$, and the two singleton unary relations $\{0\}$ and $\{1\}$ lies at precisely this node.  Its idempotent polymorphisms are exactly the term functions of the two-element semilattice.
\item CD.  The two-element template for directed unreachability is well known to lie at exactly this node.  It has three relations; the usual order relation $\{(0,0),(0,1),(1,1)\}$ on $0,1$ and the two singleton unary relations.  Its polymorphisms are exactly the term functions of the two element lattice.
\item Maltsev (or C$3$P after translation to digraphs via \cite{BDJN}).  The ternary relation on $\{0,1\}$ corresponding to the graph of addition modulo $2$, namely $\{(0,0,0),(0,1,1),(1,0,1),(1,1,0)\}$  lies exactly at this node, provided that the singleton unary relation $\{1\}$ is included to make the structure a core.  Its polymorphisms are exactly the idempotent term functions of the two-element group.  In \cite{BDJN} an equivalent template is used to produce a 78-vertex digraph which will lie at the node ``C$3$P'' in Figure \ref{fig:maltsev}.
\item ``CD and C$n$P''.  The tournament $\mathbb{T}_n$ has this property.
\item ``$\SD(\vee)$ and C$n$P''.  For $n>3$ we may take the structured union of $\mathbb{T}_n$ with $\mathbb{C}_n^{[0,2]}$: in fact the disjoint union suffices because the digraphs are homomorphism independent.
\end{itemize}

\section{Conclusions and complexity}\label{sec:separation}
Work in the present article arose partly from consideration of how various complexity classes are represented within the broad universal algebraic ``geography'' shown in Figure \ref{fig:maltsev}.  Investigation of the possible computational complexity of tractable CSPs is invited by  Allender et al.\ \cite{ABISV}, who classify complexity at the level of $2$-element templates, and given a more general footing by Larose and Tesson \cite{lartes}, who tie hardness results for a similar array of computational complexity to the omitting-type classification (the upper level of Figure \ref{fig:maltsev}).  Investigations in this direction quickly lead to consideration of the issue of stability of polymorphism properties under first order reductions.  

We say that a property $P$ is \emph{preserved under first order reductions} if whenever $P$ holds on $\CSP(\mathbb{A})$ and there is a first order reduction from $\CSP(\mathbb{B})$ to $\CSP(\mathbb{A})$, then $P$ also holds on $\CSP(\mathbb{B})$.  In general polymorphism properties are not preserved under first order reductions: the class of CSPs with $\SD(\wedge)$ polymorphisms contains problems that are $\Poly$-complete with respect to first order reductions, yet there are tractable problems that do not have the $\SD(\wedge)$ property.  The same example counterintuitively shows that natural algorithmic properties need not be preserved under first order reductions: by Barto and Kozik~\cite{b-k2,barkoz}, the $\SD(\wedge)$ property corresponds to solvability by local consistency check.  Similarly, in the present article we have seen that the strict width property (every locally consistent solution extends to a full solution) and the few subpowers algorithm are also not preserved under first order reductions.   

 At the two element level, Allender et al.~\cite{ABISV} find that all CSPs  are either first order definable, or complete in one of the classes $\Ll$, $\NL$, $\Mod_2\Ll=\oplus\Ll$, $\Poly$ and $\NP$.  In \cite{lartes} we see that in general $\Mod_p\Ll$ for prime $p$ will replace the $p=2$ case.  A complete classification of complexity of list homomorphism problems (CSPs in which all unary relations are included in the signature) over undirected graphs is obtained in Egri, Krokhin, Larose and Tesson \cite{EKLT}, where the Mod classes do not appear, but again all problems turn out be first order definable or complete in $\Ll$, $\NL$, $\Poly$ or $\NP$.  Amongst digraphs with no sources and sinks, the CSPs are either solvable in $\Ll$ or $\NP$-complete \cite{BKN}.
 In general though it seems unlikely that there is any really simple classification of complexity amongst tractable CSPs.  Even within those CSPs located at the C$n$P node, for fixed $n\geq 2$, there are $\Mod_p\Ll$-complete problems for every prime $p$, and the observations of Section \ref{sec:reddig} (such as Lemmas \ref{lem:CSPplus} and \ref{lem:modp}) show that they may be combined with structured union (or possibly direct products) to produce problems outside of $\Mod_p\Ll$ for prime $p$ and possibly even outside of $\Mod_k\Ll$ for any $k$.  Yet the polymorphism properties remain unchanged (by Theorem \ref{thm:sumclosure}). 
Nodes such as CM are even worse, because here there are problems that are unions of $\Mod_p\Ll$-complete problems (or $\cap/\cup$ combinations of $\Mod_k\Ll$ problems for varying $k$), with problems that are $\NL$-complete, or that are intersections of such languages.  However it is still possible that CSPs with the C$n$P property or the Hobby-McKenzie property can be bound within some proper subclass of \Poly.  For example, it is not out of the question that C$n$P problems lie within the complexity classes obtained by combinations of $\cap$ and $\cup$ applications to $\Mod_k\Ll$ problems for varying $k$, while 
Hobby-McKenzie problems might still also possibly lie within 
some intermediate subclass of \texttt{P}, such as \texttt{NC} for example. 
(Such speculations depend heavily on unresolved complexity theoretic issues such as the absence of a containment between the classes $\NL$ and $\Mod_k\Ll$ and $\texttt{NC}\neq \Poly$.)
 
A  more fruitful line of attack may be to attempt to show that there are no \Poly-complete problems with the C$n$P property, or with the Hobby-McKenzie property.  Again, one is pushed toward the issue of preservation of polymorphisms under first order reductions.  For example, if the C$n$P property is preserved under first order reductions then no $\Poly$-complete CSP can have the C$n$P property, and similarly for Hobby-McKenzie.

We complete the article by observing that a number of popular conjectures in this area can also be phrased in terms of the combination of a dichotomy-like conjecture and a statement about preservation of polymorphism properties under first order reductions.

\subsection{Algebraic dichotomy}  The algebraic dichotomy conjecture is equivalent to the Feder-Vardi dichotomy conjecture along with the claim that the Taylor property is preserved under first order reductions.  Indeed if the Feder-Vardi dichotomy is true (with completeness in terms of first order reductions) but the algebraic dichotomy false, then there is $\mathbb{A}$ with the Taylor property but with $\CSP(\mathbb{A})\notin\Poly$.  So $\CSP(\mathbb{A})$ is $\NP$-complete.  Consider some template $\mathbb{B}$ without the Taylor property.  By completeness, we have a first order reduction from $\CSP(\mathbb{B})$ to $\CSP(\mathbb{A})$, showing that the Taylor property is not preserved under first order reductions.

\subsection{The $\NL=\SD(\vee)$ conjecture.}
 A second widely stated conjecture is that a CSP is solvable in nondeterministic logspace if and only if it has polymorphisms witnessing the $\SD(\vee)$ property.  We mention that this $\NL=\SD(\vee)$ conjecture is implicitly premised on the assumptions that $\Mod_k\Ll\not\subseteq \NL$ and $\NL\subsetneq \Poly$; otherwise there are obvious counterexamples.
 
 If the $\NL=\SD(\vee)$ conjecture is true, then there is a dichotomy within problems of bounded width: a CSP with bounded width is either $\Poly$-complete (under first order reductions), or in $\NL$.  Indeed, if $\CSP(\mathbb{A})$ is not solvable in \NL, then if it has bounded width  it is $\SD(\wedge)$ but not $\SD(\vee)$ and so admits the semilattice type.  In this case $\CSP(\mathbb{A})$ is \Poly-complete, by the result of Larose and Tesson \cite{lartes}.
 
The $\NL=\SD(\vee)$ conjecture can now be seen as equivalent to the conjunction of the following two statements:  
 \begin{itemize}
 \item a CSP that is bounded width is either \Poly-complete (under first order reductions) or in \NL.
 \item the $\SD(\vee)$ property is preserved under first order reductions.
 \end{itemize}
 One direction of the equivalence is given above.  Now assume that the $\NL=\SD(\vee)$ conjecture fails.  So there is an $\mathbb{A}$ with the $\SD(\vee)$ property but that is not solvable in $\NL$.  As $\mathbb{A}$ does have bounded width, then either the first item fails or $\CSP(\mathbb{A})$ is $\Poly$-complete.  But then the second item fails because there are templates $\mathbb{A}$ without the $\SD(\vee)$ property but with $\CSP(\mathbb{B})$ tractable (and hence reducing to $\CSP(\mathbb{A})$).

 \subsection{The $\Ll$ conjecture.}
A third conjecture is that CSPs solvable in $\Ll$ are precisely those with the $\SD(\vee)$ and C$n$P property (for some $n$).  This conjecture is also related to statements about the preservation of polymorphism-definable properties under first order reductions.  If we assume that the $\NL=\SD(\vee)$ conjecture holds, then the conjecture on $\Ll$ gives a trichotomy for problems of bounded width: they are either $\Poly$-complete, $\NL$-complete or in $\Ll$.  Under the assumption of $\NL=\SD(\vee)$ (and the extra complexity theoretic assumptions $\Ll\subsetneq \NL$ and $\Ll\subsetneq \Mod_p\Ll$ for all $p> 1$), the $\Ll$ conjecture is equivalent to the conjunction of the following two statements.
 \begin{itemize}
 \item A CSP that is solvable in $\NL$ is either $\NL$-complete or solvable in $\Ll$.
 \item Simultaneous satisfaction of the $\SD(\vee)$ property and the C$n$P property is preserved under first order reductions (possibly for varying $n$).
 \end{itemize}

\end{document}